\documentclass[12pt,english]{amsart}
\usepackage{lmodern}
\usepackage{helvet}

\usepackage[T1]{fontenc}
\usepackage[latin9]{inputenc}
\usepackage{amsthm}
\usepackage{amstext}
\usepackage{enumerate}
\usepackage{amssymb}
\usepackage{graphicx}
\usepackage{setspace}

 \usepackage[usenames,dvipsnames,monochrome]{color}
\long\def\blue#1{{\color{blue}#1}}

\makeatletter

\numberwithin{equation}{section}
\numberwithin{figure}{section}
\newenvironment{lyxlist}[1]
{\begin{list}{}
{\settowidth{\labelwidth}{#1}
 \setlength{\leftmargin}{\labelwidth}
 \addtolength{\leftmargin}{\labelsep}
 }}
{\end{list}}
\theoremstyle{plain}
\newtheorem{thm}{\protect\theoremname}[section]
  \theoremstyle{remark}
  \newtheorem{rem}[thm]{\protect\remarkname}
  \theoremstyle{plain}
  \newtheorem{cor}[thm]{\protect\corollaryname}
  \theoremstyle{definition}
 \usepackage{amsfonts}
\usepackage{amssymb}
\usepackage[all]{xy}
\usepackage{array}
\usepackage{lmodern}
\usepackage{graphicx}
\usepackage[OT2,T1]{fontenc}

\usepackage{linearb}

\usepackage{float}
\usepackage{graphicx}

\DeclareMathAlphabet{\mathbbold}{U}{bbold}{m}{n}

\def\cB{\mathcal{B}}

\def\cX{\mathcal{X}}
\def\cV{\mathcal{V}}
\def\cY{\mathcal{Y}}
\def\cK{\mathcal{K}}

\def\cC{\mathcal{C}}
\def\cE{\mathcal{E}}

\def\cF{\mathcal{F}}
\def\cJ{\mathcal{J}}
\def\QQ{\mathbb{Q}}
\def\RR{\mathbb{R}}

\def\CC{\mathbb{C}}

\def\ZZ{\mathbb{Z}}
\def\VV{\mathbb{V}}
\def\PP{\mathbb{P}}

\def\EE{\mathbb{E}}

\def\cX{\mathcal{X}}
\def\cM{\mathcal{M}}

\def\fI{\mathfrak{I}}

\def\ul{\underline}

\def\uz{\underline{z}}

\def\uud{\underline{\underline{d}}}
\def\cO{\mathcal{O}}
\def\cD{\mathcal{D}}

\def\cS{\mathcal{S}}
\def\cN{\mathcal{N}}

\def\cB{\mathcal{B}}
\def\cW{\mathcal{W}}

\def\ay{\mathbf{i}}
\def\x{\mathsf{X}}
\def\y{\mathsf{Y}}
\def\st{\mathsf{T}}

\makeatother

\usepackage{babel}
  \providecommand{\corollaryname}{Corollary}
  
  \providecommand{\remarkname}{Remark}
\providecommand{\theoremname}{Theorem}

\usepackage{amssymb,amsfonts,amsmath,amsopn,amstext,amscd,color,latexsym,mathrsfs,skak,stackrel,url,verbatim}
\usepackage[active]{srcltx}

\theoremstyle{remark}

\newtheorem{para}{\bf}[subsection]
\newtheorem{example}[para]{\bf Example}

\newtheorem{rmk}[para]{\bf Remark}
\newtheorem{lemma}[para]{\bf Lemma}
\newtheorem{prop}[para]{\bf Proposition}
\theoremstyle{definition}

\newtheorem{defn}[para]{Definition}

 \newcommand {\C} {{\mathbb C}}
 \newcommand {\R} {{\mathbb R}}
 \newcommand {\Q} {{\mathbb Q}}

  \newcommand {\D} {{\mathcal D}}

  \newcommand {\M} {{\mathcal M}}

   \newcommand {\del} {{\partial}}
 \newcommand {\X} {{\mathscr X}}

 \newcommand {\ch}{\mathrm{CH}}
\newcommand {\chc}{\mathbf{CH}}

\newcommand{\Spec}{\text{\rm Spec}}
\newcommand{\CH}{\mathrm{CH}}
\newcommand{\ind}{\text{\rm ind}}
\newcommand{\I}{\mathbf{i}}
\newcommand{\ol}{\overline}

\newcommand{\bs}{{\backslash}}
\newcommand{\rat}{\text{\rm rat}}

\newcommand{\w}{\omega}

\newcommand{\rar}{\rightarrow}

\renewcommand{\sp}{{\rm sp}}

\begin{document}

\title[$K$-theory elevator]{\bf Specialization of cycles and the $K$-theory elevator}

\author[del \'Angel]{P. Luis del \'Angel R.}
\address{CIMAT, A.C., Callej\'on de Jalisco, S/N, Guanajuato, Gto., M\'EXICO}
\email{luis@cimat.mx}

\author[Doran]{C. Doran}
\address{Department of Mathematics, University of Alberta, Edmonton, Alberta, T6G 2G1, CANADA}
\email{charles.doran@ualberta.ca}

\author[Kerr]{M. Kerr}
\address{Department of Mathematics, Washington University in St. Louis, St. Louis, MO 63130, USA}
\email{}

\author[Lewis]{J.  Lewis}
\address{Department of Mathematics, University of Alberta, Edmonton, Alberta, T6G 2G1, KANADA}
\email{lewisjd@ualberta.ca}

\author[Iyer]{J. Iyer}
\address{The Institute of Mathematical Sciences, CIT Campus, Taramani, Chennai 600113, INDIA}
\email{jniyer@imsc.res.in}

\author[M\"uller-Stach]{S. M\"uller-Stach}
\address{Institut f\"ur Mathematik, Fachbereich 08, Johannes Gutenberg Universit\"at Mainz, 55099 Mainz, DEUTSCHLAND}
\email{mueller-stach@uni-mainz.de}

\author[Patel]{D. Patel}
\address{Department of Mathematics, Purdue University, West Lafayette, IN 47907, USA}
\email{patel471@purdue.edu}


\begin{abstract} 
A general specialization map is constructed for higher Chow groups and used to prove a ``going-up'' theorem for 
algebraic cycles and their regulators.  The results are applied to study the degeneration of the modified diagonal cycle of Gross and Schoen, and of the coordinate symbol on a genus-2 curve.
\end{abstract}

\maketitle

{\it They have ladders that will reach further, but no one will climb them.} -- A.  Sexton, ``Riding the Elevator into the Sky''

\tableofcontents

\section{\bf Introduction}

The aim of this paper is to describe limiting invariants for generalized normal functions of geometric origin at a singularity of the underlying period mapping.  To describe the underlying geometry, let $\bar{\pi}:\mathcal{X}\to\mathcal{S}$ be a proper, dominant morphism of smooth quasi-projective varieties over $\mathbb{C}$, with $\dim\mathcal{S}=1$
and smooth restriction $\pi:\mathcal{X}^{*}\to\mathcal{S}^{*}=\mathcal{S}\backslash\{s_{0}\}$.
Write $X_{s}=\bar{\pi}^{-1}(s)$, and set $\mathbb{V}:=R^{2p-r-1}\pi_{*}\mathbb{Q}(p)$,
with monodromy operator $T$ about $s_{0}$. Consider a higher Chow
cycle $\mathcal{Z}^{*}\in\mathrm{CH}^{p}(\mathcal{X}^{*},r)_{\mathbb{Q}}\cong H_{\mathcal{M}}^{2p-r}(\mathcal{X}^{*},\mathbb{Q}(r))$,
and if $r=0$ assume that the restrictions $Z_{s}=\imath_{s}^{*}\mathcal{Z}^{*}$
are homologous to zero. Then there is an associated (``higher'',
if $r>0$) admissible normal function $\nu\in\mathrm{ANF}_{\mathcal{S}^{*}}^{r}(\mathbb{V})$,
given by $\mathrm{AJ}_{X_{s}}^{p,r}(Z_{s})\in\mathrm{Ext}_{\mathrm{MHS}}^{1}\left(\mathbb{Q},H^{2p-r-1}(X_{s},\mathbb{Q}(p))\right)$
on fibers of $\pi$. 

General formulas for the regulator maps $\mathrm{AJ}^{p,r}$, first
constructed by Bloch \cite{B5}, were given in \cite{KLM}. They can
often be difficult to compute directly; even for showing that the
normal function is nonzero, one often makes do with the associated
infinitesimal invariant, inhomogeneous Picard-Fuchs equation, or (if
$r>0$) the presence of a nontorsion singularity at $s_{0}$. In the
absence of a singularity, one can also consider the \emph{limit of
the normal function} at $s_{0}$: indeed, if the cycle class $\mathrm{cl}_{\mathcal{X}^{*}}^{p,r}(\mathcal{Z}^{*})\in\mathrm{Hom}\left(\mathbb{Q},H^{2p-r}(\mathcal{X}^{*},\mathbb{Q}(p))\right)$
has vanishing residue on $X_{s_{0}}$, then $\nu$ extends to $\mathcal{S}$,
with $\nu(s_{0})$ in the generalized Jacobian of $\ker(T-I)\subseteq H_{\text{lim}}^{2p-r-1}(X_{s},\mathbb{Q}(p))$. 

A useful technique for computing this limiting value is given by \emph{specialization}:
if $\mathcal{Z}^{*}$ lifts to $\mathcal{Z}\in\mathrm{CH}^{p}(\mathcal{X},r)_{\mathbb{Q}}$,
then we obtain a class $\imath_{s_{0}}^{*}\mathcal{Z}$ in the motivic
cohomology $H_{\mathcal{M}}^{2p-r}(X_{s_{0}},\mathbb{Q}(p))$. This
formalism, and its relation to the ``naive'' specialization to $\mathrm{CH}^{p}(X_{s_{0}},r)_{\mathbb{Q}}$,
is discussed in detail in $\S$\ref{S3.0}. As a simple example, one
can think of a difference of sections of a family of elliptic curves
that degenerate to a nodal rational curve: the class of the naive
specialization is always zero, whereas the specialization into motivic
cohomology takes values in $\mathbb{C}^{*}$.

Given the specialized cycle $\imath_{s_{0}}^{*}\mathcal{Z}$, then,
we can use of a semi-simplicial hyperresolution of $X_{s_{0}}$
to compute its Abel-Jacobi class in absolute Hodge cohomology $H_{\mathcal{H}}^{2p-r}\left( X_{s_0},\mathbb{Q}(p)\right) \cong \mathrm{Ext}_{\mathrm{MHS}}^{1}\left(\mathbb{Q},H^{2p-r-1}(X_{s_{0}},\mathbb{Q}(p)\right) .$
The main general result of this paper (Theorem \ref{limThm}) is that
the image of this class under the Clemens retraction computes $\nu(s_{0})$.
Note that the case of a semistable degeneration has been treated carefully
for $r=0$ \cite{GGK}, so we concentrate in $\S$\ref{MS5} on the
\emph{higher} normal function setting, which behaves a bit differently.

The even-numbered sections are devoted to worked examples and special
cases, all of which exhibit the phenomenon referred to in the title:
this is a 7-author paper, and some of us prefer ``$K$-theory elevator'',
others ``going up''. Whatever one wishes to call it, we all felt
it merited a systematic exposition, given the many contexts in which
it arises (e.g. \cite{JW}, \cite{dS}, \cite{DoranKerr}, \cite{Ke-K3},
\cite{GGK}, \cite{Collino}). In the event that $X_{s_{0}}$ is a
normal crossing variety, and $\imath_{s_{0}}^{*}\mathcal{Z}$ ``comes
from'' its $c^{\text{th}}$ coskeleton (with desingularization $Y^{[c]}$),
the basic point is that we can interpret part of $\nu(s_{0})$ as
the regulator of a class in $\mathrm{CH}^{p}(Y^{[c]},r+c)_{\mathbb{Q}}$.
So in effect one \emph{goes up} from $K_{r}^{\text{alg}}(X_{s})$
to $K_{r+c}^{\text{alg}}(Y^{[c]})$.

The special case we study in $\S$\ref{S2} is a particular kind of
semistable degeneration, with $X_{s_{0}}$ the product of a nodal
rational curve $Q_{0}$ by a smooth variety. We briefly recall results
from \cite{KLM,KL}, and then use them to directly compute the limit
of the fiberwise regulator maps (Theorem \ref{KT}). This is applied
in $\S$\ref{toy} to compute the limit of a normal function arising
from a family of $K_{2}$ classes on elliptic curves. A related example
comes much later, in $\S$\ref{S6}, where we specialize a $K_{2}$
class on a family of genus two curves. The resulting number-theoretic
identities, \eqref{eq9} and \eqref{eq10}, had been proposed by M.
Mari\~no in recent private correspondence with two of the authors,
on the basis of the t' Hooft limit of a far-reaching conjectural relationship
between the spectrum of a quantum curve and the enumerative geometry
of its mirror \cite{Ma}.

But the motivation for this paper goes back much further, to the seminal
work of Collino \cite{Collino}, based on a fascinating idea which
he attributes to Bloch. Let $C/\mathbb{C}$ be a general genus 3 curve,
with Jacobian $J(C)$. Then the Ceresa cycle $\xi_{0}:=C-C^{-}\in\mathrm{CH}_{\text{hom}}^{2}(J(C))$
defines a non-torsion element of the Griffiths group $\mathrm{Griff}^{2}(J(C))$
\cite{Ceresa}. Collino considers a one-parameter deformation of $J(C)$,
degenerating to a singular variety ``isogenous to'' $J(D)\times Q_{0}$,
where $D$ is a general genus 2 curve. In the sense described above,
$\xi_{0}$ ``goes up'' to a $K_{1}$ class $\xi_{1}\in\mathrm{CH}^{2}(J(D),1)$,
which turns out (by an analysis of the infinitesimal invariant as
$D$ varies) to be \emph{regulator indecomposable}. This gives an
alternative proof of the nontriviality of $\xi_{0}$.

A further degeneration to $E\times Q_{0}\times Q_{0}$ (up to isogeny),
for some general elliptic curve $E$, leads (by iteration of the ``going
up'' procedure) to a non-torsion class $\xi_{2}\in\mathrm{CH}^{2}(E,2)$.
This can be identified with an Eisenstein symbol (cf. \cite[Ex. 10.1]{DoranKerr})
in the sense of Beilinson, and shown to be nontorsion in this way;
or one can argue as in \cite{Collino}. Finally, degenerating the
elliptic curve to a $Q_{0}$ leaves us with a class $\xi_{3}\in\mathrm{CH}^{3}(\text{Spec}(\mathbb{C}),2)$
(in fact defined over $\mathbb{Q}(i)$). Alternatively, one may degenerate
$C$ directly to a rational curve with three nodes and go directly
to $\xi_{3}$ as in \cite[$\S$IV.D]{GGK}, where the regulator of
this class is computed (and shown to be nontorsion) directly.

In $\S$\ref{S3}, the first step ($K_{0}\rightsquigarrow K_{1}$)
of this procedure is made much more precise, and applied to study
``going up'' for the modified (small) diagonal cycle $\Delta\in\mathrm{CH}_{\text{hom}}^{2}(C\times C\times C)$
\cite{GrossSchoen}, which is closely related to Ceresa's cycle. In
particular, we obtain a regulator indecomposable cycle in $\mathrm{CH}^{2}(D\times D,1)$,
and a new approach to the nontriviality of $\Delta$ in the Griffiths
group as a corollary (cf. Theorem \ref{thm3.6}).

A couple of comments on notation are in order. With the exception of
parts of $\S\S$\ref{S2}-\ref{S3.0}, the cycle groups in this paper
are taken with $\mathbb{Q}$-coefficients, denoted by a subscript
$\mathbb{Q}$. (This is a basic requirement for Hanamura's construction \cite{Hanamura}.) 
When describing the construction of motivic cohomology,
we also require intersection conditions on cycles (and higher cycles)
which permit them to be pulled back. In particular, if $Y_{I}=Y_{i_{1}}\cap\cdots\cap Y_{i_{\ell}}$
is a substratum of a normal crossing variety, and $Z\in Z^{p}(Y_{I},r)$
is a higher Chow precycle, we might impose the condition that $Z$
properly intersect the products of all $Y_{J}$ ($J\supset I$) and
all faces of $\square^{r}$. Such conditions will be denoted throughout
by a subscript ``$\#$'' for brevity.

\subsection*{Acknowledgments}

The authors acknowledge partial support under NSF FRG grant DMS-1361147
(Kerr), NSF grant DMS-1502296 (Patel), grants from the Natural Sciences and Engineering Research Council of Canada (Doran, Lewis), and DFG grant SFB/TRR 45 (M\"uller-Stach). We thank M. Mari\~no
for bringing \cite{Ma} to our attention, and  D. Ramakrishnan for suggesting to look at the degeneration of the modified diagonal cycle.

%%%%%%%%%%%%%%%%%%%%%%%%%%%%%%%%%%%%
\section{\bf A first view of going up: semi-nodal degenerations} \label{S2}
%%%%%%%%%%%%%%%%%%%%%%%%%%%%%%%%%%%%

We begin by providing a concrete view of ``going up'' in the very simplest setting:  that of a semi-stable degeneration with singular fiber the product of a smooth variety and a nodal rational curve.  In addition to setting the stage for $\S\S$\ref{S3.0}-\ref{S3}, this should provide the reader with some idea of how the general formulation of limiting regulators presented in $\S$\ref{MS5} was arrived at, and how to ``decrypt'' that construction.
 
\subsection{Bloch's higher Chow groups}\label{Not}
The higher Chow groups are an algebraic version of ordinary simplicial Borel-Moore homology.
Given $W/\C$ quasi-projective, let
$Z^p(W)$ denote the free abelian group generated by subvarieties of
codimension $p$  in $W$. 
Consider the ``algebraic $r$-simplex''
\[
\Delta^r = {\Spec}\biggl\{\frac{{\C}[t_0,\ldots,t_r]}
{\big(1-\sum_{j=0}^r t_j\big)}\biggr\} \simeq {\C}^r,
\]
and put
\[
Z_{\Delta}^p(W,r) =  \biggl\{\xi\in Z^p(W\times \Delta^r)\ \biggl|\ 
\begin{matrix}\xi \ \text{\rm meets\ all\ faces}\\
 \{t_{i_1} =\cdots= t_{i_\ell} = 0,\ \ell \geq 1\}\\
\ \text{\rm properly}\end{matrix}\biggr\}.
\]
Denoting by $\partial_j : Z_{\Delta}^p(W,r) \to Z_{\Delta}^p(W,r-1)$ the restriction to $j$-th facet
$t_j=0$, we note that $\del = \sum_{j=0}^r (-1)^j\partial_j : Z_{\Delta}^p(W,r)\to Z_{\Delta}^p(W,r-1)$ 
satisfies $\del^2 = 0$.
\begin{defn}
${\CH}^r(W,m) :=$ homology of $\left(Z_{\Delta}^{r}
(W,\bullet),\del\right)$ at $\bullet = m$.
\end{defn}

\subsection{Alternate take: Cubical version}
 Let $\square^r:=({\PP^1} \bs
\{1\})^r $, with coordinates $z_i$, and $\partial_{i}^{0},\ \partial_{i}^{\infty}$ 
the restriction maps to the facets $z_{i}=0,\ z_{i}=\infty$ 
respectively. The rest of the definition is completely
analogous (with $c^p(W,r)$ denoting cycles meeting all faces properly) 
except that one has to divide out degenerate cycles.
More specifically, let $ {\rm Pr}_j : \square^r \to \square^{r-1}$ be the
projection forgetting the $j$th factor. Then the degenerate cycles
are the subgroup
\[
d^p(W,r) := \sum_{j=0}^r {\rm Pr}_j^{*} \big(c^p(W,r-1)\big) \subset c^p(W,r),
\]
and we take $Z^p(W,r):=c^p(W,r)/d^p(W,r)$ with differential 
\[
\del = \sum_{j=1}^r(-1)^{j-1}\big(\del_j^0-\del_j^{\infty}\big) : Z^p(W,r) \to Z^p(W,r-1).
\]
By \cite[Thm. 4.7]{Levine2}, the simplicial and cubical complexes are quasi-isomorphic (with $\mathbb{Z}$-coefficients),
so that $$H_r\left( Z^p(W,\bullet)\right) \cong \mathrm{CH}^p(W,r).$$
\blue{
\begin{rmk}
In \cite{Hanamura}, Hanamura defines Chow cohomology groups $\chc^p(W,r)$ for quasi-projective varieties through a hypercovering, assuming resolution of singularities for varieties over the ground field. In the case of smooth varieties this coincides with Bloch's higher Chow groups. See the discussion below Remark \ref{rmk:Hanamura} for details. 
\end{rmk}
}

\subsection{The currents}
If $(z_1,...,z_r)\in \square^r$ are affine coordinates, set
\[
\mathbf{T}_r := (2\pi\I)^r T_r := (2\pi \I)^r \delta_{[-\infty,0]^r}, \;\;\;\;\;\; \Omega_r :=\int_{\square^r}\bigwedge_{j=1}^r d\log z_j,\;\;\text{and}
\]
\[
R_r := \int_{\square^r}\log z_1 \bigwedge_{j=1}^r d\log z_j -(2\pi\I)\int_{[-\infty,0]\times \square^{r-1}}\log z_2 \bigwedge_{j=3}^{r}d\log z_j
+\cdots
\]
\[
+ (-2\pi\I)^r\int_{[-\infty,0]^{r-1}\times\square^1}d\log z_r.
\]
For $\xi\in Z^p(X,r)$,
let $\pi_1 : |\xi|\subset X\times \square^r \to X$, $\pi_2 : |\xi|\subset X\times \square^r \to \square^r$.
We put 
\begin{equation} \label{Currents}
R_{\xi} = (\pi_{1,\ast}\circ \pi_{2}^{\ast}) R_r,\ 
\Omega_{\xi} = (\pi_{1,\ast}\circ \pi_{2}^{\ast})\Omega_r,\
T_{\xi} = (\pi_{1,\ast}\circ \pi_{2}^{\ast})T_r,
\end{equation}
and $\mathbf{T}_{\xi} = (2\pi \I)^r T_{\xi}$. Recall that in the Deligne cohomology complex, 
$$
{\M}^{\bullet}_{\D} = { \text{\rm Cone}\big\{{\mathcal{C}_X^{2p+\bullet}}(X,
{\ZZ}(p)) \oplus 
F^{p}\D_{X}^{2p+\bullet}(X)\to 
\D_{X}^{2p+\bullet -1}(X)\big\}[-1] ,}
$$
the differential $D$ is given by
\[
D\left( (2\pi\I)^{p-r}\left( \mathbf{T}_{\xi},\Omega_{\xi},R_{\xi}\right)\right)
= (2\pi\I)^{p-r}\left(dT_{\xi},d\Omega_{\xi},\mathbf{T}_{\xi}-\Omega_{\xi} -dR_{\xi}\right) .
\]
\[
= (2\pi\I)^{p-r+1}\left( \mathbf{T}_{\del \xi},\Omega_{\del \xi},R_{\del \xi}\right) ;
\]
the resulting cohomology at $\bullet = -r$
is $H^{2p-r}_{\D}(X,{\ZZ}(p))$. 
To guarantee that the currents in \eqref{Currents} are defined, we have to restrict to a subcomplex $Z^p_{\R}(X,\bullet)$ of cycles meeting real faces of $[-\infty,0]^m$ properly.  The main results we shall need are summarized in:
\begin{thm}
\begin{enumerate}[(i)]
\item \cite{KLM} The formula $\xi \mapsto (2\pi{\I})^{p-r}\left( \mathbf{T}_{\xi},\Omega_{\xi},R_{\xi}\right)$ induces a morphism of (cohomological) complexes
$$
Z_{\R}^{p}(X,-\bullet) \to {\M}^{\bullet}_{\D}.
$$
\item \cite{KL} The inclusion $Z_{\R}^{p}(X,\bullet) \hookrightarrow Z^{p}(X,\bullet)$ is a rational quasi-isomorphism.
\end{enumerate}
\end{thm}

In view of (ii), we shall work with higher Chow groups with $\mathbb{Q}$-coefficients $\mathrm{CH}^p(X,r)_{\mathbb{Q}}$ for the remainder of this section.

\subsection{\bf A key prototypical situation}\label{S1}

Let $\Delta \subset \C$ be a disk 
centered at $0\in \Delta$, with $\Delta^* = \Delta \bs \{0\}$, and consider the diagram
\begin{equation}\label{E1}
\begin{matrix}
X&\hookrightarrow&\ol{X}\\
f\big\downarrow\ &&\ \big\downarrow \ol{f}\\
\Delta^*&\hookrightarrow&\Delta,
\end{matrix}
\end{equation}
where $\ol{f}$ is a proper family of complex projective varieties of (relative) dimension $d$, 
and further, $f$ is smooth. This should be seen as a restriction of a global setting
\begin{equation}\label{E2}
\begin{matrix}
\X&\hookrightarrow&\ol{\X}\\
\big\downarrow&&\big\downarrow\\
B&\hookrightarrow&\overline{B},
\end{matrix}
\end{equation}
where all varieties are smooth and quasi-projective,
$\overline{B}$ is a smooth scheme of dimension 1, and $\X \to B$ is smooth and proper, with $\Delta\subset \overline{B}$ and
$\Delta^* = B\cap \Delta$.
Put $X_t = \ol{f}^{-1}(t)$, for $t\in \Delta$. Obviously $X_t$ is smooth projective for $t\in \Delta^*$, and we can consider the monodromy operator $T\in \mathrm{Aut}\left( H^{2p-r-1}(X_t)(p)\right)$. 
Let us assume that $X_0$ is reduced and of the form $Y_0 \times Q_0$, where $Y_0$ is smooth, projective, 
and $Q_0$ is a rational curve with a single node as singular set.\footnote{A similar story holds if $Q_0$ is replaced by a rational curve with multiple nodes.} In particular, $T$ is unipotent.

Now a cycle $\xi \in \CH^p(\ol{\X},r)_{\Q}$ can be assumed to meet all fibers $\{X_t\}_{t\in \Delta}$
properly; and setting $\xi_t := X_t \cdot \xi$, we will assume that $\xi_t$ belongs to $\CH_{\hom}^p(X_t,r)_{\Q}$
for $t\in \Delta$. For $t=0$, additional conditions will be imposed in $\S$\ref{Snew} below, in order that $\xi_0$ furnishes an element of Chow \emph{cohomology} of $X_0$. 

Recall that for $t\in \Delta^*$ we have the Abel-Jacobi  invariant
\[
\mathrm{AJ}(\xi_t) \in J^{p,r}\big(X_t\big) \simeq \frac{\big[ F^{d-p+1}H^{2d-2p+r+1}(X_t,\C) \big]^{\vee}}{H_{2d-2p+r+1}(X_t,\Q)(p)},
\]
given by the functional
\begin{equation}\label{ER} \omega_t \mapsto  (2\pi\I)^{r-m}\biggl( R_m(\xi_t) +
(2\pi\I)^m\int_{\del^{-1}(T_m(\xi_t))}\biggr)(\w_t).
\end{equation}
modulo periods, on test forms $\w_t\in F^{d-p+1}A_{d\text{-closed}}^{2d-2p+r+1}(X_t)$.
Here $T_r(\xi_t)$ is $T_{\xi_t} = Pr_{X_t}(\xi_t\cap \{X_t\times [-\-\infty,0]^r\})$, and $R_r (\xi_t )=R_{\xi_t}$;
writing them this way will clarify the computation below.

Consider the (co)homological situation on $X_0$.  First of all, if $p_0\in Q_0$ is the node, then
$Q_0\bs \{p_0\} = \C^*$; write $S^1$ for the unit circle.  Working with $\Q$-coefficients, we have
\[
\Q(1) \cong H^1(Q_0)(1) \underset{\longleftarrow}{\cong} H^1_c(\C^*)(1) \cong H_1 (\C^*) = \Q\langle S^1 \rangle
\]
with duals
\[
\Q(-1) \cong H_1(Q_0)(-1) \underset{\longrightarrow}{\cong} H_1^{\mathrm{BM}}(\C^*)(-1) \cong H^1(\C^*) = \Q\langle \tfrac{\mathrm{dlog}(z)}{2\pi\I}\rangle .
\]
(One may also view $(-\infty,0)$ as the generator of the untwisted Borel-Moore homology group $H_1^{\mathrm{BM}}(\C^*)$.) 
The perfect pairing
\begin{equation}\label{EQ101}
\{ H^{2p-r-2}(Y_0)(p)\otimes H^1(Q_0) \} \times \{H^{2d-2p+r}(Y_0)(d-p)\otimes H^1(\C^*)\} \to \Q
\end{equation}
may thus be interpreted via intersection or integration (on $X_0$), with the second factor  identified with a summand of \emph{homology} (of $X_0$). The plan is to view the limiting cycle $\xi_0$  as defining an element in Chow \emph{cohomology}, with Abel-Jacobi invariant in the generalized Jacobian of the \emph{first} factor of \eqref{EQ101}.

\subsection{The limiting regulator} \label{Snew}
We seek a formula for
\begin{equation}\label{limiteq}
\mathrm{AJ}(\xi_0) := \lim_{t\to 0}\mathrm{AJ}(\xi_t) \in J^{p,r}(X_0 ) ,
\end{equation}
where
\[
J^{p,r}(X_0) := \mathrm{Ext}^1_{\mathrm{MHS}}\left( \Q,H^{2p-r-1}(X_0)(p)\right) \cong \mathrm{Ext}^1_{\mathrm{MHS}}\left( \Q,\ker(T-I)(p)\right)
\]
is the ``limiting generalized Jacobian''.  (The precise sense\footnote{To give a brief glimpse of the idea: the generalized Jacobian bundle $\cup_{t\in\mathcal{S}^*} J^{p,r}(X_t)$ admits a canonical extension across the origin (cf. $\S$5.3), to which (by Theorem 5.2a) the section $\mathrm{AJ}(\xi_t)$ extends holomorphically.  The value in the fiber over the origin is what we call $\lim_{t\to 0}\mathrm{A}(\xi_t)$. This may be computed by taking limits of pairings with families of test forms representing sections of the dual canonically extended cohomology bundle (cf. Cor. 5.3).} in which the limit \eqref{limiteq} is to be interpreted is discussed in $\S$5.)  Here we are mainly interested in the K\"unneth component
\begin{equation} \label{N*}
\ul{\mathrm{AJ}}(\xi_0)\in \frac{\left(\left[ F^{d-p}H^{2d-2p+r}(Y_0,\C)\right] \otimes
H^1(\C^*,\C)\right)^{\vee}}{H_{2d-2p+r}(Y_0,\Q)(-d+p)\otimes \Q\langle S^1\rangle}
\end{equation}
corresponding to $H^1(\C^*)$ (rather than $H^0(\C^*)$).

We shall use as ``test form''
\begin{equation} \label{N**}
\w_0 = \tfrac{1}{2\pi\mathbf{i}} \eta_0 \wedge \Omega_{1} ,
\end{equation}
where $\eta_0 \in F^{d -p}A^{2d-2p+r}(Y_0,\C)$ is closed and $\Omega_1 =  \mathrm{dlog} z_1.$
Note that $\w_0$ is a limit of classes $\w_t\in F^{d-p+1}H^{2d-2p+r+1}(X_t,\C)$ as 
$t\mapsto  0$. This is a classical result stemming from an explicit description of the
canonical extension of the bundle with fibers $H^{2d-2p+r+1}(X_t,\C)$ for $t\ne 0\in \Delta$
(cf. \cite[p. 190]{Z} or \cite[III.B.7]{GGK}).

Next we impose several requirements on $\xi$ at $t=0$:  first, that $\xi$ meet properly $X_0 \times \square^r$, $\text{sing}(X_0)\times \square^r$, and all their subfaces. We can then ``naively'' define $\xi_0$ by using the canonical desingularization $\tilde{X}_0 :=Y_0 \times \mathbb{P}^1 \to Y_0 \times Q_0 \subset \overline{\mathscr{X}}$ (sending $\{0,\infty\}$ to the node $P\in Q_0$) to pull $\xi$ back to $\tilde{\xi}_0$ followed by push-forward under $\tilde{X}_0 \twoheadrightarrow X_0$ to $\mathrm{CH}^p(X_0,r)$. But this process factors through the Chow \emph{cohomology} group
\[
\chc^p(X_0,r) := H^{-m}\left\{ \mathrm{Cone} \left( Z^p (\tilde{X}_0,\bullet )\overset{\imath_0^* - \imath_{\infty}^*}{\longrightarrow}Z^p(Y_0,\bullet)\right)[-1]\right\}
\]
and the image by $\chc^p(X_0,r)\to \mathrm{CH}^p(X_0,r)$ has no invariant in \eqref{N*}. So it is appropriate to consider $\xi_0$ as an element of $\chc^p(X_0,r)$  (and thereby view $T_{\xi_0} = \mathrm{Pr}_{X_0}\left(\xi_0 \cap \{X_0 \times [-\infty,0]^m\}\right)$ in $F^r H^{2r-m}(X_0,\Q)=\{0\}$).  The general perspective will be covered in $\S\S$\ref{S3.0}-\ref{MS5}.

For the present limiting computation, we won't need the full formalism of Chow cohomology, but will rather content ourselves with the observations that $\xi_0$ defines a class in $Z^p(X_0,r)_{\del -{\rm closed}}$, as well as a class in 
$Z^p(Y_0,r+1)_{\del -{\rm closed}}$, the latter via this schema: 
\begin{multline}\label{sch}
\xi_0 \in  Z^p(Y_0\times Q_0\times \square^r) \mapsto Z^p(Y_0\times \PP^1\times \square^r)
\\
\mapsto Z^p(Y_0 \times \square^{r + 1}) \mapsto  Z^p(Y_0,r+1).
\end{multline}
In order to easily compute the regulator, we will also assume that $\xi$ and its pullbacks (to $\tilde{X}_0$, $\text{sing}(X_0)$) meet the real sub-cube faces properly (resp. those of $Y_0 \times \square^{r+1}$). Then (in view of \eqref{N**}) we have the limiting formula
\begin{equation} \label{eLAJ}
\mathrm{AJ}(\xi_t)(\w_t)\  {\buildrel {t\to 0}\over \mapsto} \ (2\pi\I)^{p-r-1}\big( R_r(\xi_0) + 
(2\pi\I)^r\delta_{\zeta}\big)(\eta_0\wedge \Omega_{1}),
\end{equation}
where $\zeta$ is a $(2d-2p+r+1)$-chain on $X_0=Y_0\times Q_0$ with $\partial \zeta = T_{\xi_0}$, properly meeting $\text{sing}(X_0)(\cong Y_0)$.\footnote{This is possible (even if $m=0$) since we assumed $\xi_0 \underset{\text{hom}}{\equiv} 0$, and $0=[T_{\xi_0}]\in H_{2d-2p+r}(X_0)$ $\implies$ $0=[T_{\xi_0}]\in H^{2p-r}(X_0)$ due to the specific form of $X_0$.}  The nodal point $p_0\in Q_0$ corresponds to $|\del [-\infty,0]|$ in the schema (\ref{sch}) above.  Let
$\zeta_0$ be a lift of $\zeta$ in $Y_0\times \square^1$.  Then
\begin{equation}\label{EE}
\del \big\{\zeta_0 \cap \{Y_0\times [-\infty,0]\}\big\} = \del\zeta_0
\cap\big\{Y_0\times [-\infty,0]\big\} \pm \zeta_0 \cap \big\{Y_0\times \del [-\infty,0]\big\},
\end{equation}
and $Pr_{Y_0}\big(\zeta_0 \cap \big\{Y_0\times \del [-\infty,0]\big\}\big) = 0$, since the lift arises from
the same copies of a membrane over a given  nodal singularity.  Therefore
\begin{equation}\label{sch1}
\del \big(Pr_{Y_0}\big(\big\{\zeta_0 \cap \{Y_0\times [-\infty,0]\}\big\}\big)\big) = 
\del\zeta_0
\cap\big\{Y_0\times [-\infty,0]\big\}.
\end{equation}
Again, via the schema (\ref{sch}) above, $\xi_0$ has a lift (which we still denote by $\xi_0$)
with support in $Y_0\times \square^{r+1}$. With the aid of (\ref{sch1}),
intersecting this lift with $Y_0\times [-\infty,0]^{r+1}$, followed by a projection to $Y_0$,
is precisely $\del \zeta_{Y_0}$, where
$\zeta_{Y_0} = Pr_{Y_0}\big(\big\{\zeta_0 \cap \{Y_0\times [-\infty,0]\}\big\}\big)$.

To compute the limiting $\mathrm{AJ}$ invariant, we shall utilize the relation of currents (cf. \cite[(5.2)]{KLM}) on $\square^n$
\[
dR_n = \Omega_n - (2\pi\I)^nT_n - 2\pi\I R_{\del \square^n}
\]
in the case $n = 1$, where it reads 
\begin{equation}\label{E01}
\Omega_1 = dR_{1} + (2\pi\I)T_{1}.
\end{equation}
In \eqref{eLAJ}, we first consider the term
\[
 \delta_{\zeta}(\eta_0 \wedge \Omega_{1} ),
\]
which by (\ref{E01}) decomposes into two pieces:
\begin{equation}\label{E3}
(2\pi \I) \delta_{\zeta} (\eta_0 \wedge T_1) \underset{\text{by }\eqref{sch1}}{=} (2\pi \I) \delta_{\zeta_{Y_0}}(\eta_0) ;
\end{equation}
and
\[
\delta_{\zeta} (\eta_0 \wedge d[R_1] ) = (-1)^r  \delta_{\zeta}(d[\eta_0 \wedge R_1]) ,
\]
which by Stokes's theorem\footnote{we are also using the general fact that $R_n$ vanishes along $(\mathbb{P}^1)^n\setminus \square^n = \bigcup_{j=1}^n\PP^1\times\cdots\times \{1\}\times \cdots \times \PP^1 \subset [\PP^1]^{\times n}$, which here is just the vanishing of $R_1 = \log z$ at $1$.}
\begin{equation}\label{eB}
= (-1)^r T_{\xi_0} (\eta_0 \wedge R_1) = (-1)^r  \left( (T_r \wedge R_1)(\xi_0) \right) (\eta_0)
\end{equation}
Recalling the relation $(-2\pi\I)^r T_r \wedge R_{1} + R_r\wedge \Omega_{1} = R_{r+1}$ from \cite{KLM}, the remaining part of \eqref{eLAJ}
\begin{equation}\label{E4}
( R_m(\xi_0))(\eta_0 \wedge \Omega_{1}) = \left( (R_m \wedge \Omega_1)(\xi_0)\right) (\eta_0)
\end{equation}
now combines with $(2\pi\I)^r \eqref{eB}$ to yield simply
\[
 R_{r+1}(\xi_0)(\eta_0),
\]
so that altogether \eqref{eLAJ} becomes
\[
(2\pi\I)^{p-r-1} \left\{ R_{r+1}(\xi_0)(\eta_0) + (-2\pi\I)^{r+1} \int_{\zeta_{Y_0}} \eta_0 \right\} \underset{\text{pds.}}{\equiv} \mathrm{AJ}(\xi_0)(\eta_0).
\]

Summarizing, we have

\begin{thm}\label{KT}
Given the above setting of subsection \ref{S1} of a normal function induced by 
\[
\mathrm{AJ}(\xi_t) \in \blue{J^{p,r}\big(X_t\big)},
\]
where $t\in \Delta^*$, $\xi_t\in \CH^p_{\hom}(X_t,r)_{\Q}$, and where $X_0 = Y_0\times Q_0$,
then 
\[
\lim_{t\to 0}\mathrm{AJ}(\xi_t)(\w_t) = \mathrm{AJ}(\xi_0)(\eta_0),
\]
where $\xi_0$ is interpreted as defining a class in $\CH^p(Y_0,r+1)_{\Q}$.
\end{thm}

\begin{rmk} \label{RR} {(i)} The situation $X_0 = Y_0 \times Q_0$ can be replaced by
$Y_0\times Q_0^{\ell}$ ($Y_0$ smooth) for $\ell\geq 1$, and a parallel analysis expresses the limiting regulator as the regulator of a class in $\mathrm{CH}^p(Y_0,r+\ell)_{\Q}$.  But there is a caveat in order
here:  the total space $\ol{X}$  over $\Delta$ cannot  be both smooth and semistable if $\ell>1$.
It all boils down to the situation $V(x_1y_1-t,...,x_Ny_N-t) \subset \CC^{2N} \times\Delta$,
a variety  which is singular at $(0,...,0)$ if $N>1$. This can be remedied in a number of ways: by blowing up (along the lines of $\S$\ref{S3.2}), allowing non-semistable degenerations (cf. $\S$\ref{sec1.3}), or by passing to several variables (viz., 
$V(x_1y_1-t_1,...,x_Ny_N-t_N) \subset \CC^{2N} \times\Delta^N$; not pursued here).

{(ii)} Many natural moduli spaces do not contain singular fibers of the form $X_0 = Y_0 \times Q_0$.
For instance, let $Z\subset \PP^5$ be a very general hypersurface of high degree. Then $Z$ does not contain any rational curves, and hence neither does any hyperplane section $X_0$ of $Z$. Furthermore, there are Hodge-theoretic obstructions to having such a degeneration. This is another reason to develop the more general perspectives in $\S\S$\ref{S3.0} and \ref{MS5}.
\end{rmk}

\subsection{A toy model} \label{toy}

Let $\pi: X\to \PP^{1}$ be the elliptic surface defined by
$$
y^{2} = x^{3} + x^{2} + t =: h(x),
$$
and let $\Sigma =\{0,\infty,\tfrac{4}{27}\}\subset \PP^{1}$ denote
the singular set of $\pi$. (Note that $X_0$ and $X_{\frac{-4}{27}}$ are nodal curves, while $X_{\infty}$ is a simply-connected tree of $\PP^1$'s. We wish to verify, as a first application of Theorem \ref{KT}, that $\mathrm{CH}^2(X_t,2)_{\Q} \neq \{0\}$ for very general $t\in \mathbb{P}^1$. Of course, this is a known fact in view of

\begin{thm} \cite{lew2,As} Let $U = X\bs \big\{X_{0},X_{\frac{-4\ }{27}},X_{\infty}
\big\}$. Then 
$$
\Gamma\big(H^{2}(U,\Q(2))\big) \simeq \Q^{2};
$$
moreover it is generated by $[\Omega_{\xi '}],[\Omega_{\xi ''}]$,
where
\[
\xi ' = \biggl\{\frac{(y-x)^{3}}{8},\frac{(y+x)^{3}}{8}\biggr\}
\biggl\{\frac{y+x}{y-x},t\biggr\}^{3},\
\]
\[
\xi '' = \biggl\{\frac{({\I}y +x+\tfrac{2}{3})^{3}}{8},
\frac{({\I}y-x-\tfrac{2}{3})^{3}}{8}\biggr\}
\biggl\{\frac{{\I}y-x-\tfrac{2}{3}}{{\I}y+x+\tfrac{2}{3}},-t-\tfrac{4}{27}\biggr\}^{3},
\]
are classes in $\CH^{2}(U,2;\Q)$.

\end{thm}

Indeed, given any class $\xi\in \CH^2(U,2)$ such that
$[\Omega_{\xi}]$ is nonzero in $\Gamma (H^{2}(U,\Q(2)) )$, standard arguments (injectivity of the topological invariant)
imply that $\mathrm{AJ}(\xi_t)$ (hence $\ch^2(X_t,2)$) is nontorsion for very general $t$.

For the approach based on limits, take a small disk $\Delta$  centered at $t=0$. For $t\in \Delta^*$, $\xi_t '' $ belongs to $\CH^2(X_t,2)_{\Q}$, and for $t=0$, we shall interpret $\xi_0 ''$ as an element of $\ch^2(\Spec(\C),3)_{\Q}$. We attend to several details.
First, $X_0 = V(y^{2} = x^{3} + x^{2})$ is a nodal rational curve parameterized by $\mathbb{P}^1/\{0,\infty\}$ via
$$z\mapsto \left(\frac{4z}{(z-1)^2},\frac{4z(z+1)}{(z-1)^3}\right) = \left( x(z),y(z)\right) .$$
The restriction of $\xi''$ to $X_0$ may be written
$$\xi_0 '' = 9\left( z,\frac{3}{2^{\frac{5}{3}}}\left( -\I y(z)+x(z)+\tfrac{2}{3}\right),\frac{-\I y(z)+x(z)+\frac{2}{3}}{\I y(z) + x(z) +\frac{2}{3}}\right) ,$$
as a cycle in $\square^3$, and we set
$$ w(z):=\frac{3}{2^{\frac{5}{3}}}\left(-\I y(z) + x(z) + \tfrac{2}{3}\right) .$$
Write $\gamma$ for the closed path $T_z = [-\infty,0]$ on $X_0$; and note that, on $\gamma$, $w(z)$ winds once clockwise about $0$.  Moreover one easily sees that
\begin{equation} \label{E:bound}
2^{-\frac{5}{3}} \leq \left. |w| \right|_{\gamma} \leq 2^{-\frac{2}{3}}
\end{equation}
and
$$ \left. \frac{-\I y + x + \frac{2}{3}}{\I y + x + \frac{2}{3}}\right|_{\gamma} = \left. \frac{w}{\bar{w}}\right|_{\gamma} .$$
So along $\gamma$, $\xi_0 ''$ looks like $(z,w,\tfrac{w}{\bar{w}})$, and $\log(\tfrac{w}{\bar{w}})$ is zero at $\gamma\cap T_w = \{ w=-\tfrac{1}{3}\}$.

For the regulator, then,
\begin{flalign*}
R& := \mathrm{AJ}(\xi_0 '')(1) = \tfrac{1}{2\pi \I} \int_{\xi_0 ''} R_3 \\
&= 9 \int_{\gamma} \log(w)\mathrm{dlog}(\tfrac{w}{\bar{w}}) \\
&= 18\I \int_{\gamma} \log (w) \mathrm{darg}(w) \\
\implies \mathrm{Im}(R) &= 18\int_{\gamma} \log|w| \mathrm{darg}(w).
\end{flalign*}
Using the bounds \eqref{E:bound} and reversing the path (for a positive measure), we conclude that
\begin{equation} \label{E:bound2}
36\pi \cdot \tfrac{2}{3}\log(2) \leq \mathrm{Im}(R)\leq 36\pi \cdot \tfrac{5}{3} \log(2).
\end{equation}
Consequently we have

\begin{thm}
\[ 
\mathrm{AJ}(\xi_0)\ne 0 \in H_{\D}^1(\Spec(\C),\Q(2)) \simeq \frac{\C}{\Q\cdot \pi^2}.
\]
\end{thm}

\begin{rem}
From a different point of view, limiting calculations were performed in \cite[$\S$6.3]{DoranKerr} for several families of elliptic curves. The case related to the present calculation is the ``E8'' curve family
$$ E_{\st} : \;\; \x\y = \st \left( 1+\x^2 + \y^3 \right) , $$
which is birational to a base change of the Tate curve via
$$\Theta:\; (\x,\y,\st)\mapsto \left( -(2\st)^2 \y , (2\st)^3 \x - (2\st)^2 \y, -(2\st)^6 \right) = (x,y,t).$$
The symbol studied in [op. cit.] is $\{\x,\y\}=\tfrac{1}{18}\Theta^* \xi'$; and there is a birational automorphism $\alpha: (x,y,t)\mapsto \left(-x-\tfrac{2}{3},\I y,-t-\tfrac{4}{27}\right)$ of the Tate curve with $\alpha^* \xi'=\xi''$. Overall, $\alpha^{-1}\circ \Theta$ sends the fiber $E_{4^{-\frac{1}{3}}3^{-\frac{1}{2}}}=:E_{\st_0}$ isomorphically to $X_0$, and pulls $\xi''$ back to $\{\x,\y\}^{18}$. Modulo a \emph{conjectural} relation in the Bloch group, it was shown in [op. cit.] that $\tfrac{1}{2\pi\I} \mathrm{AJ} \left( \{\x,\y\}_{E_{\st_0}} \right) = \tfrac{10}{3\pi} G$, where 
\begin{equation} \label{eqG}
G:=\sum_{n\geq 0}(-1)^n (2n+1)^{-2} = L(\chi_4,2)
\end{equation}
is Catalan's constant. So this would give that $\mathrm{Im}(R) = 120 \cdot G$, which agrees with \eqref{E:bound2} above.
\end{rem}

\subsection{Speculation}
As another application of the semi-nodal instance of the going-up principle, we briefly address
a relationship between the Griffiths group of
a threefold and the group of indecomposables on a given surface. 

Begin with a diagram
 \[
 \begin{matrix} \mathscr{X}&\hookrightarrow&\ol{\mathscr{X}}\\
 f\big\downarrow\quad &&\quad \big\downarrow\ol{f}\\
 B&\hookrightarrow&\ol{B}
 \end{matrix}
 \]
 where $\ol{\mathscr{X}}$ is a smooth projective fourfold,  $\ol{B}$ is a smooth
 projective curve and $f$ is smooth and proper. Put $X_t := f^{-1}(t)$,
 a smooth {threefold}. A cycle $\xi\in \CH^2(\overline{\mathscr{X}})$ which is relatively homologous 
 \blue{to zero} determines a normal function
 \[
 \nu_{\xi} : B \to \coprod_{t\in B(\C)}\blue{J^{2,0}\big(X_t\big)},
\]
with topological invariant $[\nu_{\xi}] \in {\rm Hom}_{{\rm MHS}}(\Q, \big(H^1(B,R^3f_*\Q(2))\big)$.
When this is nonzero, then under suitable monodromy conditions, $\mathrm{Griff}^2(X_t)_{\Q}\neq \{0\}$ 
for very general $t\in B(\C)$.

Now consider the situation where for some $0\in \ol{B}\bs B$,
$X_0 = Y_0\times Q_0$. Viewing $\xi_0$ as a class in $\mathrm{CH}^2(Y_0,1)_{\Q}$, we may ask whether it is \emph{indecomposable}, i.e. nonzero in $\mathrm{CH}^2_{\text{ind}}(Y_0,1)_{\Q}/(\mathrm{CH}^1(Y_0)\otimes\C^* )$. A stronger condition is \emph{regulator indecomposability}, which is to say that $\mathrm{AJ}(\xi_0)$ is nonzero in $J^{2,1}(Y_0)/\left(\mathrm{NS}(Y_0)\otimes \C^*\right)$.

The point is that the limiting Abel-Jacobi calculation (Theorem \ref{KT}) gives a connection between these conditions on $\nu_{\xi}$ and $\xi_0$. First note that for very general $t\in \Delta^*$,
 $N^1H^3(X_t,\Q(2))$ has constant rank. One has a map
 \[
 \mathrm{Griff}^2(X_t) \to J\biggl(\frac{H^3(X_t,\Q(2))}{N^1H^3(X_t,\Q(2))}\biggr).
 \]
There are natural isomorphisms
\[
N^1H^3(X_t) \simeq N^1H^3(X_t)^{\vee}, \quad  [N^1H^3(X_t)]^{\perp} \simeq \big([N^1H^3(X_t)]^{\perp}\big)^{\vee},
\]
and so
\[
J\big([N^1H^3(X_t,\Q(2))]^{\perp}\big) \simeq \frac{\big([N^1F^2H^3(X_t,\C)]^{\perp}\big)^{\vee}}{[N^1 H_3(X_t,\Q(2))]^{\perp}}.
\]
At $t=0$, a similar calculation holds, leading to a specialized analogue of Theorem \ref{KT}, where
the limiting calculation is of the form
\[
\mathrm{AJ}(\xi_t) \in J\big([N^1H^3(X_t,\Q(2))]^{\perp}\big) \mapsto \mathrm{AJ}(\xi_0)\in J\big(H^2_{\rm tr}(Y_0,\Q(2))\big).
\]
A well-known conjecture (see \cite{dJ-L}) states that 
\[
\mathrm{AJ} : \CH_{\ind}^2(Y_0;1;\Q) \to J\big(H^2_{\rm tr}(Y_0,\Q(2))\big),
\]
is injective. Assuming this, we have a diagram
\begin{equation}\label{EL}
\small \xymatrix{
\{\xi_t\}\in {\rm Griff}^2(X_t;\Q) \ar @{-->} [d]_{\text{(?)}} \ar [r]^{\mathrm{AJ}(\xi_t)\mspace{110mu}} & J\big([N^1H^3(X_t,\Q(2))]^{\perp}\big) \simeq J\biggl(\frac{H^3(X_t,\Q(2))}{N^1H^3(X_t,\Q(2))}\biggr) \ar @{~>} [d]^{\lim_{t\to 0}}
\\
\{\xi_0\}\in \CH^2_{\ind}(Y_0,1;\Q) \ar @{^(->} [r]^{\mathrm{AJ}(\xi_0)\mspace{80mu}} &J\big(H^2_{\rm tr}(Y_0,\Q(2))\big)
\simeq J\biggl(\frac{H^2(Y_0,\Q(2))}{N^1H^2(Y_0,\Q(2))}\biggr)
} \normalsize
\end{equation}
where the limiting map (?) is defined by making the diagram commutative. In particular, then, we expect that indecomposability of $\xi_0$ implies nontriviality of $\xi_t$ in the Griffiths group.  This line of inquiry, as well
as various generalizations,\footnote{both to higher degrees of $K$-theory and to higher $\mathrm{AJ}$ maps and the Bloch-Beilinson filtration \cite{lew}. Note that we do not see a way to define the dotted arrow without assuming injectivity of the bottom Abel-Jacobi map.} will be pursued in a later work.

On the other hand, there is nothing at all conjectural about \emph{regulator} indecomposability of $\xi_0$ implying nontriviality of $\xi_t$ in the Griffiths group (for $t$ general).  This will be spelled out in the worked example of $\S$4 (see Theorem \ref{thm3.6}), for which we shall need the slightly more general language of the next section.

\section{\bf Motivic picture: Specialization and going-up}\label{S3.0}
In this section, we recall the construction of specialization maps in the settings of higher Chow groups and motivic cohomology, and prove some elementary properties. These results are then applied to articulate a more general perspective on ``going up'' in $K$-theory.

\subsection{Specialization for Higher Chow groups}

In the following, $ f: X \rightarrow B$ will denote a flat morphism of regular noetherian (equi-dimensional) schemes where $B= Spec(R)$ is the spectrum of a discrete valuation ring. In this setting, 
Levine (\cite{Levine}) has defined a theory of higher Chow groups $\CH_{d+r-p}(X,r)\cong \CH^{p}(X,r)$ ($d=$ relative dimension of $f$). The $\CH_q(X,r)$ are defined as the homology groups of a certain complex $\blue{Z_q(X, \bullet)}$. These satisfy the following properties: 
\begin{enumerate}
\item If $X$ and $B$ are essentially of finite type over a field $k$, then these are the usual higher Chow groups defined by Bloch.
\item If $Z\subset X$ is a closed (pure codimension) subscheme (of finite type over $B$) of codimension $c$, then there is a long exact localization sequence
$$ \rightarrow {\rm CH}^{p-c}(Z,r) \rightarrow {\rm CH}^{p}(X,r) \rightarrow {\rm CH}^{p}(X \setminus Z, r) \xrightarrow{\partial} \blue{{\rm CH}^{p-c}(Z,r-1)} \rightarrow .$$ 
\end{enumerate}

\begin{rmk}
In our applications, we work in the setting of a degenerating family over a one-dimensional base $B$ of equi-characteristic zero.
\end{rmk}

\noindent Let $\pi$ be a fixed uniformizer in $R$, $s$ denote the closed point of $B$, and $\eta$ denote the generic point. Furthermore, let $X_{s}$ (resp. $X_{\eta}$) denote the corresponding special (resp. generic) fiber; note that by virtue of regularity of $X$, $X_{\eta}$ is smooth. Let $f_{s}$ (resp. $f_{\eta}$) denote the restriction of $f$ to the special fiber (resp. generic fiber). Finally, let $i: X_{s} \hookrightarrow X$ and $j: X_{\eta} \hookrightarrow X$ denote the natural inclusions. Then $\psi : = f_{\eta}^{*}(\pi) \in \CH^{1}(X_{\eta}, 1)$ and one can define a specialization map
\begin{equation}\label{ESP}
Sp_{\pi}: \CH^{p}(X_{\eta}, r) \rightarrow \CH^{p}(X_{s}, r).
\end{equation}
by setting $Sp_{\pi}(y) := \partial(\psi \cdot y)$, where $\partial: \CH^{p+1}(X_{\eta}, r+1) \rightarrow \CH^{p}(X_{s}, r)$ is the boundary map coming from the localization sequence. Note that pullback morphisms induce a $\CH^{*}(X, *)$-module structure on both $\CH^{*}(X_{\eta}, *)$ and $\CH^{*}(X_{s}, *)$. Moreover, since the localization sequence respects the module structure, the boundary map $\partial$ is a morphism of $\CH^{*}(X,*)$-modules. It follows that $Sp_{\pi}$ is also compatible with this module structure.

\begin{rmk}
(1) If $n=0$, these specialization maps are already considered in Fulton (\cite{Fulton}). In this case, the morphisms are independent of the choice of uniformizer, and preserve ring structures. In particular, $Sp_{\pi}: \CH^{*}(X_{\eta}) \rightarrow \CH^{*}(X_{s})$ is a ring homomorphism. \\
(2) If $X = B$, then the specialization morphisms above were considered by Bloch \cite[$\S$5.2]{Bloch4}. It is shown there that, under the additional assumption that $B$ contains its residue field, the specialization map is an algebra map.\end{rmk}

It is likely that the construction of the specialization map and the following properties are known to the experts. However, we give the details here due to the lack of a reference.

\begin{prop}\label{Sp} 

\begin{itemize}
\item[(1)] With notation as above, the following diagram commutes: 
$$
\xymatrix{
\CH^{p}(X,r) \ar[d]^{i^{*}} \ar[r]^{j^{*}}   & \CH^{p}(X_{\eta}, r) \ar[dl]^{Sp_{\pi}} \\
 \CH^{p}(X_{s}, r)  & . }
$$
\item[(2)] Let $g: X \rightarrow X'$ denote a proper morphism of regular schemes smooth over $B$. Then the following diagram commutes:
$$
\xymatrix{
\CH_{q}(X_{\eta}, r) \ar[r]^{Sp_{\pi}} \ar[d]^{g_{\eta*}} & \CH_{q}(X_{s}, r) \ar[d]^{g_{s*}} \\
\CH_{q}(X'_{\eta}, r) \ar[r]^{Sp_{\pi}}  & \CH_{q}(X'_{s}, r).  }
$$
\item[(3)] Let $g: X \rightarrow X'$ denote a flat morphism of regular schemes smooth over $B$ which is equi-dimensional of relative dimension $d$. Then the following diagram commutes:
$$
\xymatrix{
\CH_{q}(X'_{\eta}, r) \ar[r]^{Sp_{\pi}} \ar[d]^{g_{\eta}^{*}} & \CH_{q}(X'_{s}, r) \ar[d]^{g_{s}^{*}} \\
\CH_{q+d}(X_{\eta}, r) \ar[r]^{Sp_{\pi}}  & \CH_{q+d}(X_{s}, r).  }
$$
\item[(4)] Let $i: Z \subset X$ denote a regular (codimension $c$) immersion \blue{with smooth generic fiber} over $B$. Then the following diagram commutes:
$$
\xymatrix{
\CH_{q}(X_{\eta}, r) \ar[r]^{Sp_{\pi}} \ar[d]^{i_{\eta}^{*}} & \CH_{q}(X_{s}, r) \ar[d]^{i_{s}^{*}} \\
\CH_{q-c}(Z_{\eta}, r) \ar[r]^{Sp_{\pi}}  & \CH_{q-c}(Z_{s}, r).  }
$$

\item[(5)] Let $\zeta \in \CH^{p}(X_{\eta}, 1)$. If $\zeta$ is decomposable, then $Sp_{\pi}(\zeta)$ is decomposable.  
\end{itemize}
\end{prop}
\begin{proof}
1: Given $y \in \CH^{p}(X,n)$, one has $$Sp_{\pi}(j^{*}(y)) = \partial(j^{*}(y) \cdot \psi) = i^{*}(y) \partial(\psi) = i^{*}(y).$$
2: This follows from an application of the projection formula combined with the fact that $\partial$ commutes with push-forward. Namely, let $f': X' \rightarrow B$ denote the structure map and $\psi' := f_{\eta}^{'*}(\pi)$. Note that $g_{\eta}^{*}(\psi' )= f_{\eta}^{*}(\pi) = \psi$. One has:  $g_{s*}(Sp_{\pi}(z)) = $
$$
=g_{s*}(\partial(z \cdot \psi)) = \partial(g_{\eta*}(z \cdot g_{\eta}^{*}(\psi'))) = \partial(g_{\eta*}(z) \cdot \psi')= Sp_{\pi}(g_{\eta*}(z)).
$$
3: This follows from the fact that pull-back is a ring homomorphism. Namely,
$$g^{*}_{s}(Sp_{\pi}(z)) = g_{s}^{*} (\partial(z \cdot \psi')) = \partial(g_{\eta}^{*}(z \cdot \psi')) = \partial(g_{\eta}^{*}(z) \cdot \psi) = Sp_{\pi}(g^{*}_{\eta}(z)). $$
4: The proof is the same as in Part (3).\\
5: Recall,  by defintion:
$$\CH^{p}_{dec}(X,1)= Im(\CH^{1}(X,1) \otimes \CH^{p-1}(X) \rightarrow \CH^{p}(X,1)).$$
Let $\zeta \in \CH^{p}(X_{\eta},1)$ be a decomposable element. Since specialization is additive,
it suffices to prove the result for $z$ which is the image of a tensor $\zeta_{1} \otimes \zeta_{2}$ for $\zeta_{1} \in \CH^{1}(X_{\eta},1)$
and $\zeta_{2} \in \CH^{p-1}(X_{\eta})$. Note that $\zeta_{2}$ can be lifted to an element $\tilde{\zeta_{2}} \in \CH^{p-1}(X)$. Since specialization is compatible with $\CH^{*}(X,*)$-module structure, one has
$$Sp_{\pi}(\zeta) = Sp_{\pi}(\zeta_{1} \cdot \zeta_{2}) = \tilde{\zeta_{2}}Sp_{\pi}(\zeta_{1}) = Sp_{\pi}(\zeta_{2}) \cdot Sp_{\pi}(\zeta_{1}).$$
It follows that $Sp_{\pi}(\zeta)$ is decomposable.
\end{proof}

\begin{rmk} Note that proof of Part (2) above does not require the smoothness of $f$ or $f'$, only that the generic fibers are smooth. The analogous remark also applies to Part (3).
\end{rmk}

\begin{rmk}
The last part of Proposition \ref{Sp} was proved by Collino and Fakhruddin (\cite{Co-Fd}, Theorem 2.1) under the assumption that the cycle $\zeta$ lifts to $X$.  The proof here also partially applies to $\CH^{p}(X_{\eta}, r)$. Namely, the same proof shows that if an element of $\CH^{p}(X_{\eta}, r)$ lies in the image of $\CH^{r}(X,r) \otimes \CH^{p-r}(X_{\eta})$ (whenever this makes sense), then
 the same can be said of its specialization.
\end{rmk}

Note that $Sp_{\pi}$ depends on the choice of uniformizer 
in the setting of higher Chow groups. However, one has the following comparison result.

\begin{lemma}
With notation as above, let $\pi' = u \pi$ be another choice of uniformizer where $u$ is a unit in $R$. Then $Sp_{\pi'}(a) = Sp_{\pi}(a) + (-1)^{r}(u\partial(a))$
for any $a \in \CH^{p}(X_{\eta}, r)$.
\end{lemma}
\begin{proof}
This follows directly from the fact that the boundary maps $\partial$ in the localization sequence are $\CH^{*}(X,*)$-module maps. 
\end{proof}

\begin{rmk} \label{rmk:Hanamura}
We note that on $ker(\partial: \CH^{p}(X_{\eta},r) \rightarrow \CH^{p-1}(X_{s},r-1))$, the specialization map is independent of the choice of uniformizer. This follows from Part (1) of the previous proposition (or from the Lemma).
\end{rmk}

We conclude this section by noting that the results of this section also pass to motivic cohomology. We refer to (\cite{Hanamura}) for the basic definition and construction of motivic cohomology. Here (passing to $\Q$-coefficients)
we simply recall some of the properties.\\

(1) Given any quasi-projective variety $S$ over a field $k$ of characteristic zero (or more, generally characteristic $p$, assuming resolution of singularities) one can associate to it the 
Chow cohomology groups $\chc^{p}(S,r)$. Briefly, these are defined by choosing a semi-simplicial hyper-resolution $X^{\bullet} \rightarrow S$, and then taking the total complex of the double complex formed by the Bloch higher cycle complex associated to the corresponding semi-simplicial scheme. It can be shown that the construction is independent of the chosen hyper-resolution. We refer to (\cite{Hanamura}) for the details.\\
(2) The Chow cohomology groups come equipped with a contravariant functoriality (for arbitrary maps) and a ring structure.\\
(3) These are covariantly functorial under proper maps with smooth target, and under flat maps of projective varieties. \\
(4) They agree with the usual higher Chow groups in the smooth case. \\

Suppose now we have a $f: X \rightarrow B$ as before, where $X$ is regular, and $f$ is proper and generically smooth. Suppose further that we are in the equi-characteristic zero case. In this case, $X_{\eta}$ and $X$ are smooth. The previously stated properties of motivic cohomology allow one to specialize cycles on $X_{\eta}$ which are liftable to $X$. For usual cycles, one has a diagram 
$${\chc}^{p}(X_{s}) \xleftarrow{i^{*}} {\rm CH}^{p}(X) \twoheadrightarrow {\rm CH}^{p}(X_{\eta}).$$
We may lift a cycle $\zeta \in {\rm CH}^{p}(X_{\eta})$, and then pull-back to the motivic cohomology group. In general, this `specialization' depends on the lift. However, in the following we shall work with examples that come equipped with canonical extensions to $X$.\footnote{One should be aware that even (or perhaps especially) in this situation, properties such as cohomological or algebraic equivalence to zero on nearby fibers need not specialize.} Similarly, for higher cycles one has a diagram:
$${\chc}^{p}(X_{s},r) \xleftarrow{i^{*}} {\rm CH}^{p}(X,r) \twoheadrightarrow ker(\partial: {\rm CH}^{p}(X_{\eta},r) \rightarrow {\rm CH}^{p-1}(X_{s},r-1)).$$
In particular, if we are given natural extensions of classes $\zeta$ in the right-most term to all of $X$, then we can specialize them to the motivic cohomology of $X$. 
These constructions are functorial in families. Namely, suppose we are given two families $f: X \rightarrow S$ and $f': X' \rightarrow S$, as above. Suppose, moreover that we have a proper $S$-morphism
$F: X \rightarrow X'$ of relative dimension $c$. Then we have a natural commutative diagram:
$$
\xymatrix{
{\chc}^{p}(X_{s})  \ar[d] & {\rm CH}^{p}(X) \ar[l]  \ar[r] \ar[d] & {\rm CH}^{p}(X_{\eta}) \ar[d] \\
{\chc}^{p-c}(X'_{s})   & {\rm CH}^{p-c}(X')  \ar[l] \ar[r] & {\rm CH}^{p-c}(X'_{\eta}) }
$$
Here the vertical maps are given by push-forward. 
\begin{rmk}
(1) In the following subsection, our cycles will be naturally liftable to $X$, and the previous method combined with the descent spectral sequence will allow one to construct higher Chow cycles on singular strata of the special fiber.\\
(2) One could also work with the motivic cohomology of Suslin and Voevodsky; indeed, it is known that $\chc^p(X,n)\cong H^{2p-n}_{\cM}(X,\Q(p))$. However, in the following we shall use convenient hyper-resolutions (in the spirit of Hanamura and Levine) to explicitly compute motivic cohomology.
\end{rmk}

\subsection{Examples of going-up for algebraic cycles} \label{S3ex}
We now demonstrate how to use the specialization map to produce a ``going-up'' calculus for higher Chow cycles, which will be elaborated in $\S$\ref{MS5}. Namely, we show that in 
certain types of degenerations, the specialization morphisms combined with edge morphisms in a certain descent spectral sequence allows one to construct higher weight Chow cycles from lower weight cycles.

Working over a field of characteristic zero, we continue to assume that $X$ is regular, and $f$ generically smooth; write $\dim(X)=d+1$. In this setting, we have constructed specialization morphisms:
$$Sp_{\pi}:\, \CH^{p}(X_{\eta}, r) \rightarrow \CH^{p}(X_{s}, r),$$
$$\imath^*:\,\mathrm{CH}^p(X,r)\to \chc^p(X_s,r).$$
Of course, we can compose $Sp_{\pi}$ with the restriction to obtain a map
$$sp:\, \mathrm{CH}^p(X,r)\to \mathrm{CH}^p(X_s,r)$$
that is independent of $\pi$.

Suppose we are given a smooth proper semi-simplicial hypercover $\mathfrak{X}^{\bullet} \rightarrow X_{s}$. In this setting, one has a 
($1^{st}$ quadrant) descent spectral sequence:
\begin{equation} \label{p17*}
\mathrm{E}^{1}_{\ell,k}(q):= \CH_{q}(\mathfrak{X}^{\ell},k) \Rightarrow \CH_{q}(X_{s},\ell+k) .
\end{equation}
(See for example \cite[Thm. 1.4]{Ge}; this also follows from the double complex for Chow homology in \cite[Def. 2.10]{Hanamura}, by taking the associated spectral sequence \cite[$\S$5.6]{Wei}.) More importantly, one has similar spectral sequence in the setting of motivic cohomology. In this case, one has (associated to the Chow cohomology double-complex in \cite[Def. 2.10]{Hanamura}) a $4^{th}$ quadrant cohomological spectral sequence:
\begin{equation} \label{p17**}
\mathrm{E}_{1}^{\ell, k}(p):= \CH^{p}(\mathfrak{X}^{\ell},-k) \Rightarrow \chc^{p}(X_{s}, -(\ell+k)).
\end{equation}
Rewriting \eqref{p17*} as a $3^{rd}$ quadrant cohomological spectral sequence ${}'E_1^{\ell,k}(p):= E^1_{-\ell,-k}(d-p)$, there is an obvious map $E_1^{\bullet,\bullet}(p)\to {} 'E_1^{\bullet,\bullet}(p)$ given by the identity on the $(0,k)$-entries and by zero elsewhere. This induces a homomorphism $\chc^p(X_s,r)\overset{\theta}{\to}\mathrm{CH}^p(X_s,r)$ factoring $sp =\theta \circ \imath^*$. However, $\theta$ tends to lose much of the information we want to understand in the limit (via $\imath^*$).

\begin{example}\label{ex:AL}
We now apply this to the simple situation of a semi-nodal degeneration, to give the abstract perspective on $\S$\ref{S2}. Write $X_s = Y\times Q$, with $Q$ a nodal rational curve.  In this case, a smooth hypercover can be constructed by taking the usual normalization. Then $\mathfrak{X}^{0} = Y \times \PP^{1} \rightarrow Y \times Q$ is given by identity on the first component and is just the normalization on the second component. Moreover, $\mathfrak{X}^{1} = Y$ and the semi-simplicial scheme $\mathfrak{X}^{\bullet} \rightarrow X_{s}$ is a proper smooth hypercover. In this setting, the 4th-quadrant descent spectral sequence for motivic cohomology has two columns. Moreover, the differentials on the $E_{1}$-terms are given by the difference of pullbacks via $i_{0}, i_{\infty}: Y \rightarrow Y \times \mathbb{P}^{1}$. Since this difference is zero, the descent spectral sequence \emph{degenerates}. In particular, one has a natural map $$\chc^{p}(X_{s}, r) \rightarrow \CH^{p}(Y,r+1),$$
which does \emph{not} factor through $\theta$. One can now compose this with the pull-back map, to get a going-up map:
$$\CH^{p}(X, r) \rightarrow \CH^{p}(Y, r+1).$$
In particular, given an extension of a cycle on the generic fiber to all of $X$, one can specialize it to a higher Chow cycle on $Y$. 

Again we emphasize that $\text{im}(\sp)\subseteq\text{im}(\theta)$, where $\theta$ is a motivic analogue of taking the ``image of cohomology in homology''.  Often this simply has the effect of killing everything. For example, if $p=r=2$ and $Y=\mathrm{Spec}(F)$ is a point over a number field, then $\chc^2(X_s,2)\cong \mathrm{CH}^2(F,3)\cong K_3^{\text{ind}}(F)_{\mathbb{Q}}$ while $\mathrm{CH}^2(X_s,2)\cong \mathrm{CH}^1(F,1)\cong K_1(F)_{\mathbb{Q}}$. In this scenario, we have $\text{image}(\sp)=\{0\}=\text{image}(\theta)$.  So only $\imath^*$ (and not $\sp$) captures the $K_3^{\text{ind}}$ information in the limit.
\end{example}
 
Typically one cannot expect the descent spectral sequence to degenerate at $E_1$. In order to formulate more general ``going-up'' statements, we introduce a filtration, writing
$$
\mathscr{W}_{-b} \chc^p(X_s,r) \subset \chc^p (X_s,r)
$$
for the image of the cohomology of $E_1^{\ell\geq b,k}(p)$.

\begin{example}\label{specializationGSexam}
 
 One can apply a similar argument in the setting of degenerations of triple products of curves. Namely, suppose we are in a setting where 
 $F: {}'\mathcal{C} \rightarrow B$ is a semistable family of genus 3 curves, and let ${}'\mathcal{X} := {}'\mathcal{C} \underset{F}{\times} {}'\mathcal{C} \underset{F}{\times}{}'\mathcal{C}$
 denote the triple fiber-product. Suppose that the special fiber ${}'\mathcal{C}_{s} = \widetilde{\cC_s}  \cup \mathbb{P}^{1}$ where $\widetilde{\cC_s}$ is the normalization of 
 an irreducible curve $\cC_s$ of arithmetic genus three with one node. Moreover, in that case, $\widetilde{\cC_s}$ is a smooth hyperelliptic curve of genus 2, and we assume that the 
 inverse image of the node consists of the two Weierstrass points on $\widetilde{\cC_s}$. Finally, suppose $\widetilde{\cC_s} \cap \mathbb{P}^{1}$ consists precisely of these two Weierstrass points. 
 In this setting, Gross and Schoen \cite{GrossSchoen} have constructed a good family $f: \mathcal{X} \rightarrow B$ such that $f$ is flat, proper, and the total space is smooth. 
 Moreover, the generic fiber $\mathcal{X}_{\eta} = {}'\mathcal{X}_{\eta}$, and the special fiber $\mathcal{X}_{s}$ has eight components (cf. $\S$\ref{S3.2}).

 In the next section, we shall study the modified diagonal cycle (cf. $\S$\ref{S3.1})  in $\CH^{2}(C_{\eta} \times C_{\eta} \times C_{\eta})$, which has a natural extension
to $\mathcal{X}$. The previous constructions then allow one to specialize the modified diagonal to a cycle in $\mathscr{W}_{-1} \chc^{2}(\mathcal{X}_{s}).$
Furthermore, the previous description of the components of $\mathcal{X}_{s}$ give rise to a natural smooth proper hypercover of 
$\mathcal{X}_{s}$. Considering the associated descent spectral sequence as in the previous example gives rise to edge maps
\begin{equation}\label{eqnsp}
\mathscr{W}_{-1} \chc^{2}(\mathcal{X}_{s}) \rightarrow \CH^{2}(C' \times C',1).
\end{equation}
It follows that the image of the specialization of the modified diagonal under the image of this map gives rise to a higher Chow cycle 
in $\CH^{2}(C' \times C',1)$, and in what follows we shall make the relation of this degeneration and the Abel-Jacobi map precise.
 \end{example}

 \section{\bf Degeneration of a  modified diagonal cycle}\label{S3}

In this section, we provide details on the Example sketched in \S \ref{specializationGSexam}. Furthermore, we show that the specialization is an indecomposable higher Chow cycle.

%%%%%%%%%%%%%%%%%%%%%%%%%%%%%%%%%%%%%%%%%%%%%%%%%%%%%%%  %%%%%%%%%%%
\subsection{Modified diagonal cycle on a triple product of a curve} \label{S3.1}
%%%%%%%%%%%%%%%%%%%%%%%%%%%%%%%%%%%%%%%%%%%%%%%%%%%%%%%%%%%%%%%%%

Given a smooth projective curve $C$ of genus $g$ (defined over $\CC$), the \emph{modified diagonal cycle} of Gross and Schoen \cite{GrossSchoen} on $X:=C\times C\times C$ can be described as follows. Fixing a closed point $e\in C(\CC)$, consider the codimension-$2$ subvarieties 
\begin{eqnarray*}
 \Delta_{123} & := & \{x,x,x): x\in X\}\\
 \Delta_{12} & := & \{(x,x,e): x\in X\} \\
 \Delta_{13} & := & \{(x,e,x): x\in X\} \\
 \Delta_{23} & := & \{(e,x,x): x\in X\}\\
 \Delta_1 & := & \{(x,e,e): x\in X\} \\
 \Delta_2 & := & \{(e,x,e): x\in X \} \\
 \Delta_3 & := & \{(e,e,x): x\in X\}
 \end{eqnarray*}
of $X$; then the cycle
\begin{equation} \label{eq3(1)}
\Delta_e := \Delta_{123} -\Delta_{12}-\Delta_{13}-\Delta_{23} + \Delta_{1}+ \Delta_2 + \Delta_3 \in Z^2(X)
\end{equation} 
is homologous to zero \cite[Prop. 3.1]{GrossSchoen}.  Furthermore:
\begin{itemize}
\item if $g_C = 0$, then $\Delta_e \underset{\rat}{\equiv} 0$; and
\item if $C$ is hyperelliptic, then $6\Delta_e \underset{\rat}{\equiv} 0$ \cite[Prop. 4.8]{GrossSchoen}.
\end{itemize}
For each $p\in C(\CC)$, we have Abel maps
$$
\begin{array}{cccc}
\varphi_p^{\pm} : & C &\rar& J(C) \\
& q & \mapsto &\pm \mathrm{AJ}(q-p)\\
\end{array}
$$
with image $C_p^{\pm}=\varphi_p^{\pm}(C)$, and 
\begin{equation}\label{eq3(1.5)}
\begin{array}{cccccc}
f:&X & \rar & Sym^3 C &\rar & J(C)\\
& (q_1,q_2,q_3) &\mapsto & \Sigma q_i & \mapsto & \mathrm{AJ}(\Sigma q_i - 3p) \\
\end{array}
\end{equation}
Recall that the Ceresa cycle is defined by
$$
Z_{C,p} := C_p^+ - C_p^- \in Z_{\text{hom}}^{g-1} (J(C)) ;
$$
when we consider it in $\mathrm{Griff}^{g-1}(J(C)) = Z_{\text{hom}}^{g-1}(J(C))/Z^{g-1}_{\text{alg}}(J(C)),$ where it is nontorsion for $C$ general (in particular, non-hyperelliptic), we may drop the ``$p$''.  The same goes, of course, for the subscripts on $f_p$ and $\Delta_e$. According to results of Colombo and van Geemen \cite[Props. 2.9 and 3.7]{Colombo-van-Geemen}, in $\mathrm{Griff}^{g-1}(J(C))$ we have
\begin{equation} \label{eq3(2)}
f_* \Delta \underset{\text{alg}}{\equiv} 3 Z_C
\end{equation}
whenever $C$ is hyperelliptic or trigonal -- in particular, if $g_C = 3$.  Furthermore, we have the following:

\begin{lemma} \label{lem3.1}
If $g_C =3$, then $f^* f_* \Delta \underset{\text{alg}}{\equiv} 6\Delta$ (in $\mathrm{Griff}^2(C^{\times 3})$).
\end{lemma}

\begin{proof}
In fact, we claim that for $p=e$, $f^* f_* \Delta = 6\Delta$ in $Z^2(C^{\times 3})$. Indeed, this formula holds for the morphism $f':C^{\times 3} \to \mathrm{Sym}^3 C$ by \cite[(4.4)]{GrossSchoen}. Now write $f=h\circ f'$, where $h:\mathrm{Sym}^3 C \to \mathrm{Pic}^3 C \cong J(C)$. Here $\mathrm{Pic}^3 C$ is the degree-3 Picard scheme, with the isomorphism given by $e$; and $h$ is a birational morphism, namely the blow-up of $\mathrm{Pic}^3 C$ along the curve $-C+\omega_C = \{ \omega_C (-x) \mid x\in C \} \subset \mathrm{Pic}^3 C$ (cf. \cite[p. 360, Ex. 2(b)]{Bi-La}). As the support of $f_* ' \Delta_e$ does not lie in the  exceptional locus of the blow-up morphism, we have
$h^*h_*(f'_*(\Delta_e))\,=\, f'_*(\Delta_e)$; and so
\begin{eqnarray*}
f^*f_*(\Delta_e) &=& f'^* h^*(h_*(f'_*(\Delta_e)) \\
                 & =& f'^*f'_*(\Delta_e)\\
                 &=& 6 \Delta_e
\end{eqnarray*}
as desired.
\end{proof}

Together with \eqref{eq3(2)}, the Lemma implies that for $C$ of genus 3, we have (in $\mathrm{Griff}^2 (C^{\times 3})$)
\begin{equation} \label{eq3(2.5)}
f^* Z_C \underset{\text{alg}}{\equiv} 2\Delta .
\end{equation}

In what follows, we shall explain how to use the behavior of $Z_C$ under degeneration to understand that of $\Delta$.  (We shall also take $p=e$.)

%%%%%%%%%%%%%%%%%%%%%%%%%%%%%%%%%%%%%%%%%%%%%%%%%%%%%%%%%%%%

\subsection{Degeneration of $C^{\times 3}$ and $J(C)$}\label{S3.2}
%%%%%%%%%%%%%%%%%%%%%%%%%%%%%%%%%%%%%%%%%%%%%%%%%%%%%%%%%%%%

Let $\cC\to \mathrm{Spec}(R)=:B$ be a (flat, proper) family of stable curves over a DVR, with regular total space. The Jacobian $J(\cC_{\eta})$ of the (smooth) generic fiber (over $\eta = B\setminus \{s\}$) is extended over $B$ by the N\'eron model $N_g(\cC/B)$, whose special fiber is a finite disjoint union of semi-abelian varieties \cite{BLR}. One completion (to a proper $B$-scheme) is given by the moduli scheme $\bar{P}_g (\cC/B)$ of degree $g$ semibalanced line bundles, which contains $N_g(\cC/B)$ as a dense open subscheme \cite{Ca-Es}. Write $N_g(\cC_s) \subset \bar{P}_g (\cC_s)$ for the special fibers.

On the other hand, if $'\cC$ is a semistable family and the components of $'\cC_s$ are smooth, Gross and Schoen construct a ``good model'' $\cX \to B$ for $'\cC \underset{B}{\times} {}'\cC \underset{B}{\times} {}'\cC$. In particular, $\cX$ is flat and proper over $B$, with regular total space, such that $\cX_{\eta} = {}'\cC_{\eta}^{\times 3}$.

The particular case of interest for us is when $\cC$ has genus $g=3$, and $\cC_s$ is irreducible, with one node $q$. Then $\cJ := \bar{P}_3 (\cC/B)$ is smooth (over $\mathbb{C}$); and one may describe the special fiber $\cJ_s = \bar{P}_3(\mathcal{C}_s)$ as follows.  First observe that its normalization $\widetilde{\cJ}_s$ is a $\mathbb{P}^1$-bundle over $\tilde{A} :=J(\widetilde{\cC_s})$. Then $\cJ_s$ is formed by  attaching the $0$- and $\infty$-sections of this bundle with a shift by $\varepsilon := \mathrm{AJ}_{\widetilde{\cC_s}}(\tilde{q}_2 - \tilde{q}_1 )\in \tilde{A}(\mathbb{C})$, where $\{ \tilde{q}_1, \tilde{q}_2 \} \subset \widetilde{\cC_s}$ lie over $q$. This shift records the Hodge-theoretic extension class of
\begin{equation} \label{eq3(3)}
0\to H_1 (\tilde{A}) \to H_1 (\cJ_s) \to H_1 (\bar{\mathbb{G}}_m) \to 0 ,
\end{equation}
where $\bar{\mathbb{G}}_m := \mathbb{P}^1/\{0,\infty\}$ is the nodal rational curve.\footnote{\eqref{eq3(3)} is obtained by identifying the end terms of $0\to H_1(\widetilde{\cJ_s})\to H_1(\cJ_s)\to H_0(\tilde{A})\to 0$ with $H_1(\tilde{A})$ and $H_1(\bar{\mathbb{G}}_m)$, respectively; the second identification seems like a cheap trick (both are $\QQ(0)$ as Hodge structures), but is natural once we make the 2-torsion assumption below (which yields a projection from $\cJ_s$ to $\bar{\mathbb{G}}_m$).}  The open smooth subset $\cJ_s^* = N_3 (\cC/B) \subset \cJ_s$ is itself an extension of $\tilde{A}$ by $\mathbb{G}_{m,\mathbb{C}}$; the corresponding extension of Hodge structures
\begin{equation} \label{eq3(4)}
0\to H_1 (\mathbb{G}_m) \overset{\imath}{\to} H_1 (\cJ_s^* )\overset{\tilde{\rho}}{\to} H_1(\tilde{A}) \to 0
\end{equation}
is (by the first bilinear relation) dual to \eqref{eq3(3)}. Henceforth we shall take $\{ \tilde{q}_1, \tilde{q}_2\}$ to be Weierstrass points on $\widetilde{\cC_s}$, so that \eqref{eq3(3)} and \eqref{eq3(4)} are $2$-torsion extensions of MHS. In this case, there exists a homomorphism $\sigma: \cJ_s^* \to \mathbb{G}_m$ with $(\sigma\circ\imath)(z)=z^2$, so that $\tilde{\rho}\times\sigma:\,\cJ_s^* \twoheadrightarrow \tilde{A}\times \mathbb{G}_m$ is a 2:1 isogeny. Writing $\rho$ for the composition of $\tilde{\rho}$ with $\tilde{A}\overset{\text{2:1}}{\twoheadrightarrow} A:=\tilde{A}/\langle \varepsilon \rangle$, $\rho \times \sigma$ extends to a map
\begin{equation} \label{eq3(4.5)}
\rho: \cJ_s \twoheadrightarrow A\times \bar{\mathbb{G}}_m =: \mathcal{A}
\end{equation}
which is 4:1 on $\cJ_s^*$ (and 2:1 on $\mathrm{sing}(\cJ_s)\cong \tilde{A}$). Write $\cJ^{\bullet}_s \to \mathcal{A}^{\bullet}$ for the map of semi-simplicial schemes, where $\cJ_s^0 = \tilde{\cJ}_s$, $\cJ^1_s = \mathrm{sing}(\cJ_s) = \tilde{A}$ (resp. $\mathcal{A}^0 = A\times \mathbb{P}^1$, $\mathcal{A}^1 = A$).

Now our chosen $\cC$ doesn't satisfy the hypotheses of \cite{Ca-Es}: the sole component of $\cC_s$ is singular. To fix this, we take the base change of $\cC$ under $t\mapsto t^2$ ($B\to B$) and blow up the double point to get $'\cC \to B$ semistable, with $'\cC_s = \widetilde{\cC_s} \cup \mathbb{P}^1$ ($\widetilde{\cC_s}\cap \mathbb{P}^1 = \{\tilde{q}_1,\tilde{q}_2\} = \{0,\infty \}$). The special fiber of the associated good model $\cX$ is $\cX_s = \cup_{i=1}^8 Y_i$, where \cite[Ex. 6.15]{GrossSchoen}:
\begin{itemize}
\item $Y_2$ (resp. $Y_3, Y_4$) is the blow-up of $\mathbb{P}^1 \times \widetilde{\cC_s} \times \widetilde{\cC_s}$ (resp. $\widetilde{\cC_s} \times \mathbb{P}^1 \times \widetilde{\cC_s}$, $\widetilde{\cC_s}\times \widetilde{\cC_s} \times \mathbb{P}^1$) along the $\left\{ \mathbb{P}^1 \times \{ \tilde{q}_i \} \times \{ \tilde{q}_j \} \right\}$;
\item $Y_5$ (resp. $Y_6,Y_7$) is the blow-up of $\widetilde{\cC_s} \times \mathbb{P}^1 \times \mathbb{P}^1$ (resp. $\mathbb{P}^1 \times \widetilde{\cC_s} \times \mathbb{P}^1$, $\mathbb{P}^1 \times \mathbb{P}^1 \times \widetilde{\cC_s}$) along the $\left\{ \widetilde{\cC_s} \times \{ \tilde{q}_i \} \times \{ \tilde{q}_j \} \right\}$ ;
\item $Y_1 \twoheadrightarrow \widetilde{\cC_s} ^{\times 3}$ (resp. $\{ \tilde{q}_i \} \times \mathbb{P}^1 \times \mathbb{P}^1$), $\tilde{\mathbb{P}}^2$($=$ degree-6 del Pezzo)-fibers over the 8 points $\{ \tilde{q}_i \} \times \{ \tilde{q}_j \} \times \{ \tilde{q}_k \}$, and point fibers elsewhere.
\end{itemize}
We will write $\cX^{\bullet}_s$ for the corresponding semi-simplicial scheme, where $\cX^{\ell}_s := \coprod_{|I|=\ell + 1} Y_I$ ($I\subset \{ 1,\ldots ,8\}$, $Y_I := \cap_{i\in I} Y_i$).

%%%%%%%%%%%%%%%%%%%%%%%%%%%%%
\subsection{Extension of the Abel map} \label{S3.3}
%%%%%%%%%%%%%%%%%%%%%%%%%%%%%

Likewise, we can base-change the extended Jacobian $\cJ$ (via $t\mapsto t^2$) and blow up the preimage of $\tilde{A}$; this results in a smoth total space $'\cJ$ and singular fiber $'\cJ_s = {} '\cJ_{s,0} \cup {} '\cJ_{s,1}$ ($'\cJ_{s,i} \cong \widetilde{\cJ_s}$), where $'\cJ_{s,0}$ is the ``identity'' component.

Fix a section $\underline{e}: B\to {} '\cC$ such that $\underline{e}_s$ is a Weierstrass point on $\widetilde{\cC_s}\subset {} '\cC_s$, distinct from $\tilde{q}_1$ and $\tilde{q}_2$. Together with \eqref{eq3(1.5)}, this yields a map $\cX_{\eta} \overset{F_{\eta}}{\to} {}'\cJ_{\eta}$ over $\eta$, which extends continuously to a well-defined morphism 
$$
F:\, \cX \to {}'\cJ.
$$
On the smooth locus $\cX^{\text{sm}}_s = \left( ( \cC_s \setminus \{q\}) \cup \mathbb{G}_m \right)^{\times 3}$ of the singular fiber $\cX_s$, this extension may be described Hodge-theoretically, or alternatively (at least on $(\cC_s\setminus \{q\} )^{\times 3}$) by pulling back the Abel-N\'eron map of \cite{Ca-Es}. Explicitly, we send $(p_1,p_2,p_3)\mapsto \sum_{i=1}^3 \int_{\underline{e}_s^{\epsilon_i}}^{p_i} \in \omega({}'\cC_s )^{\vee}/H_1({}'\cC_s ) \cong \cJ_{s,|\underline{\epsilon}|}^* ,$ where $\underline{e}_s^0 := \underline{e}_s$, $\underline{e}^1_s := 1\in \mathbb{G}_m$, $|\underline{\epsilon}|:=\sum \epsilon_i$ (mod $2$), and $\epsilon_i = 0$ (resp. $1$) if $p_i \in \cC_s \setminus \{q\}$ (resp. $\mathbb{G}_m$). In particular, $Y_1,Y_5,Y_6,Y_7$ are mapped to $'\cJ_{s,0}$ while $Y_2,Y_3,Y_4,Y_8$ go to $'\cJ_{s,1}$.

Below we shall only need the composition
$$
\pi :\, \cX \to \cJ
$$
of $F$ with the finite morphism $'\cJ \twoheadrightarrow \cJ$ of degree 2. On the singular fiber, the composition $\rho \circ \pi_s : \cX_s \to \mathcal{A} (= A\times \bar{\mathbb{G}}_m)$ is easy to describe: $Y_2,Y_3,Y_4,Y_6$ are collapsed to $\mathrm{sing}(\mathcal{A})$; $Y_5,Y_6,Y_7$ have 2-dimensional image; $Y_1(\twoheadrightarrow \widetilde{\cC_s}^{\times 3})\to(\tilde{A}\twoheadrightarrow)A$ is the $\mathrm{AJ}$ map for the genus 2 (hyperelliptic) curve $\widetilde{\cC_s}$; and $Y_1(\twoheadrightarrow \cC_s^{\times 3})\to \bar{\mathbb{G}}_m^{\times 3} \overset{\times}{\to}\bar{\mathbb{G}}_m$ is the product of the hyperelliptic maps on factors. Our situation is summarized by the diagram
\[
\xymatrix{
\cX \ar @/^1em/  [rr]^{\pi} \ar [r] & {} '\cJ \ar [r] & \cJ
\\
\cX_s \ar @{^(->} [u]^{\imath_{\cX}} \ar [r] \ar @/_1em/ [rr]_{\pi_s} & {} '\cJ_s \ar @{^(->} [u] \ar [r] & \cJ_s \ar @{^(->} [u]_{\imath_{\cJ}} \ar [r]_{\rho} & \mathcal{A} .
}
\]
 
%%%%%%%%%%%%%%%%%%%%%%%%%%%%%%%%%
\subsection{Extension and specialization of cycles}\label{S3.4}
%%%%%%%%%%%%%%%%%%%%%%%%%%%%%%%%%

The choice of $\underline{e}$ gives us a natural family of modified diagonal cycles on $\cX_{\eta}$ and Ceresa cycles on $\cJ_{\eta}$; the naive extensions (obtained by taking closures of each irreducible component $\Delta_i,\Delta_{ij},\Delta_{ijk},C^+,C^-$) will be denoted  by $\Delta = \Delta_{\underline{e}}\in \mathrm{CH}^2(\cX)$ and $Z_{\cC}=Z_{\cC,\underline{e}}\in \mathrm{CH}^2(\cJ)$. We may consider the specializations $\imath^*_{\cX}\Delta \in \mathbf{CH}^2(\cX_s)$ and $\imath_{\cJ}^* Z_{\cC}\in \mathbf{CH}^2 (\cJ_s)$ in motivic cohomology. The idea is then that if these are cohomologically trivial in $H^4(\cX_s)$ resp. $H^4(\cJ_s)$, we expect they are rationally equivalent to zero (with $\mathbb{Q}$-coefficients) on the normalizations $\cX^0_s$ resp. $\cJ_s^0$,\footnote{in view of the triviality ($\otimes \mathbb{Q}$) of Ceresa cycles and modified diagonal cycles for hyperelliptic curves (hence for the genus-$2$ curve $\widetilde{\cC_s}$).} which would allow us to ``go up'' into (subquotients of) $\mathbf{CH}^2(\cX_s^1,1)$ resp. $\mathbf{CH}^2 (\cJ_s^1,1)$.  In view of \cite[Prop. III.B.9]{GGK} or \cite[Prop. 7.2]{GrossSchoen}, this cohomological triviality holds after replacing $\Delta$ resp. $Z_{\cC}$ by a modification of the form $\hat{\Delta} := \Delta  - (\imath_{\cX})_* \cW_{\Delta}$ ($\cW_{\Delta} \in Z^1 (\cX_s)$) resp. $\hat{Z}_{\cC} := Z_{\cC} - (\imath_{\cJ})_* \cW_{Z}$ ($\cW_Z \in Z^1 (\cJ_s )$).\footnote{In fact, for codimension-2 cycles this can be accomplished integrally, after multiplying the original cycle by the exponent of the (finite) singularity group $$G:=\frac{\mathrm{im}\{H_4 (\cX_s,\mathbb{Q})\to H_4(\cX,\mathbb{Q})\}_{\mathbb{Z}}}{\mathrm{im\{H_4(\cX_s,\mathbb{Z})\to H_4(\cX,\mathbb{Z})\}}}.$$}

Since $\imath^*_{\cX}\Delta$ is nonzero on each component $Y_i\subset \cX_s$, the direct construction of $\cW_{\Delta}$ becomes a complicated exercise in intersection theory and combinatorics. Instead we shall proceed indirectly, using the fact that $\imath^*_{\cJ}Z_{\cC}$ is \emph{already} cohomologically trivial. Here it is convenient to use $\rho$; while $\rho$ is not flat, we can construct an \emph{ad hoc} push-forward map, $\mathbf{CH}^2(\cJ_s)\overset{\rho_*}{\to}\mathbf{CH}^2(\mathcal{A})$ by the map of double complexes $Z^2_{\#}(\cJ_s^{\bullet},-\bullet) \to Z^2_{\#}(\mathcal{A}^{\bullet},-\bullet)$ given by $\rho_*$ on $\cJ_s^0$ and $2\rho_*$ on $\cJ_s^1$. Then we have $\rho^*\rho_* = 4\cdot\mathrm{Id}$ on $\mathbf{CH}^2(\cJ_s)$, and
$$
\rho_* Z_{\cC,s} =: Z^+_A - Z^-_A \in Z^2(\mathcal{A}^0)=Z^2(A\times \mathbb{P}^1)
$$
is evidently rationally equivalent to zero. Indeed, writing $z_A :\widetilde{\cC_s} \to \mathbb{P}^1$ for the hyperelliptic map and $\phi_A^{\pm}$ for the composition $\widetilde{\cC_s} \underset{\varphi_{\underline{e}_s}^{\pm}}{\to} \tilde{A} \underset{\rho}{\twoheadrightarrow} A,$
$$
Z_A^{\pm} = \left( \phi_A^- \times z_A^{\pm 1}\right) (\widetilde{\cC_s}) = \left( \phi_A^+ \times z_A^{\pm 1}\right) (\widetilde{\cC_s}) \subset A\times \mathbb{P}^1
$$
may be viewed as the graph of $z_A^{\pm 1}$ over the nodal curve $\phi_A^+ (\widetilde{\cC_s})$($\cong \cC_s$). (Moreover, the zero and pole of $z_A$ are located at the node.) The rational equivalence is given by the push-forward of $\frac{z-z_A}{z-z_A^{-1}}$ under $\widetilde{\cC_s}\times \mathbb{P}^1 \underset{\phi_A^+ \times \mathrm{Id}}{\to} A\times \mathbb{P}^1$, whose divisor is precisely $Z_A^+ - Z_A^-$. Viewing this pushforward as an element of $Z^2 _{\#}(\mathcal{A}^0,1)$ from $\rho_* Z_{\cC,s} \in Z^2_{\#}(\mathcal{A}^0)$ yields
\begin{equation}\label{eq3**}
Z_{\cC}^{(1)} := \left( \phi_A^+ (\widetilde{\cC_s}),z_A^2 \right) \in \ker(\partial)\subset Z^2(\mathcal{A}^1,1) = Z^2 (A,1).
\end{equation}

By the projective bundle formula, $\mathrm{CH}^2(A\times \mathbb{P}^1 ,1) \underset{\imath_0^* - \imath_{\infty}^*}{\to} \mathrm{CH}^2(A,1)$ and $\mathrm{CH}^2(\cJ_s^0,1) \underset{\imath_0^* - \imath_{\infty}^*}{\to}\mathrm{CH}^2(\tilde{A},1)$ are zero; we conclude:

\begin{prop} \label{prop3.4}
The specialization $\imath_{\cJ}^* Z_{\cC}$ of the Ceresa cycle, belongs to $\mathscr{W}_{-1} \chc^2 (\cJ_s) = \rho^* \mathscr{W}_{-1} \chc^2(\mathcal{A})$($=\mathrm{CH}^2(\tilde{A},1) = \rho^* \mathrm{CH}^2 (A,1)$), and is represented by $Z^{(1)}_{\cC}$.
\end{prop}

%%%%%%%%%%%%%%%%%%%%%%%%%%%%%%%%%
\subsection{Indecomposability of the specialization} \label{S3.5}
%%%%%%%%%%%%%%%%%%%%%%%%%%%%%%%%%

Recall the higher Abel-Jacobi maps associated to this situation:
$$
\xymatrix{
\mathrm{CH}^2(A,1) \ar [r]^{\mathrm{AJ}^{2,1}} \ar @{->>} [d] & J\left( H^2(A,\mathbb{Q}(2)) \right) \ar @{->>} [d]
\\
\mathrm{CH}^2_{\text{ind}} (A,1) \ar [r]^{\overline{\mathrm{AJ}}^{2,1}} & J\left( H^2_{\text{tr}}(A,\mathbb{Q}(2))\right)
}
$$
where $J(H):=\mathrm{Ext}^1_{\text{MHS}}(\mathbb{Q},H)=\frac{H_{\mathbb{C}}}{\left\{ F^0 H_{\mathbb{C}} + H_{\mathbb{Q}} \right\} }$, and $\mathrm{AJ}^{2,1}(Z)$ ($Z=(C,\phi)$) is given by the class of the current $2\pi \mathbf{i} \int_C (\log \phi)(\,\cdot\,) + (2\pi \mathbf{i})^2 \int_{\Gamma}(\,\cdot\,)$ (where $\partial\Gamma = \phi^{-1}(\mathbb{R}_-)$). We say that $Z$ is \emph{regulator indecomposable} if $\overline{\mathrm{AJ}}^{2,1}(Z)\neq 0$; by the diagram, this implies indecomposability.

\begin{prop}
For $\widetilde{\cC_s}$ very general in the moduli space $\cM_2$ of genus 2 curves, $Z_{\cC}^{(1)}$  is regulator indecomposable. (Hence for $\widetilde{\cC_s}$ general, $Z_{\cC}^{(1)}$ is indecomposable.)
\end{prop}

\begin{proof}
$Z_{\cC}^{(1)}$ is a multiple of Collino's cycle; apply the main result of \cite{Collino}.
\end{proof}

By \eqref{eq3(2.5)}, $\tfrac{1}{2} \pi^* Z_{\cC} =: \tilde{\Delta}$ is algebraically equivalent to $\Delta$ on the generic fiber. To describe the precise sense in which 
\begin{equation} \label{eq3(!!)}
\imath_{\cX}^* \tilde{\Delta} = \tfrac{1}{2} \imath^*_{\cX} \pi^* Z_{\cC} = \tfrac{1}{2} \pi_s^* \imath_{\cJ}^* Z_{\cC} \in \mathscr{W}_{-1}\mathrm{CH}^2(\cX_s)
\end{equation}
remains regulator indecomposable, we look at the spectral sequence $E_0^{a,b} = \oplus_{|I|=a+1} Z^2 (Y_I , -b)_{\#}$ computing $\mathrm{CH}^2(\cX_s)$ ($d_0=\partial$, $d_1 = \delta$). Let $\left( \mathrm{Gr}^{\mathscr{W}}_{-1} \chc^2 (\cX_s)\right)_{\text{ind}}$ denote the quotient of 
$$
\mathrm{Gr}^{\mathscr{W}}_{-1}\chc^2(\cX_s) = \left\{ \ker(d_1)\cap\ker(d_2) \subset \tfrac{\oplus \mathrm{CH}^2(Y_{ij},1)}{\delta \left( \oplus \mathrm{CH}^2(Y_i,1)\right)} \right\}
$$
by the subspace of (equivalence classes of) decomposable cycles; further, $\mathscr{S}_3$ acts on $\cX_s$, and we let $(\cdots)^{\mathscr{S}_3}$ denote invariants.

\begin{lemma} \label{lem3.5}
We have isomorphisms
\begin{enumerate}[(a)]
\item $\left( \mathrm{Gr}^{\mathscr{W}}_{-1}\mathbf{CH}^2(\cX_s)\right)^{\mathscr{S}_3}_{\text{ind}} \cong \mathrm{CH}^2_{\text{ind}}(\widetilde{\cC_s}\times \widetilde{\cC_s} ,1)^{\mathscr{S}_2}$ 

\noindent and
\item $\left(\mathrm{Gr}^W_2  H^3 (\cX_s) \right)_{\text{tr}}^{\mathscr{S}_3} \cong H^2_{\text{tr}}(\widetilde{\cC_s}\times \widetilde{\cC_s})^{\mathscr{S}_2} \overset{\cong}{\underset{(\pi_s^{(1)})^*}{\leftarrow}} H^2_{\text{tr}}(\tilde{A}).$
\end{enumerate}
\end{lemma}

\begin{proof}
First note that $\mathrm{CH}^2_{\text{ind}}(Y_{ij},1)$ is zero for all but $Y_{12},Y_{13},Y_{14}$, each of which has two components (because of $\tilde{q}_1,\tilde{q}_2$). Moreover, we can ignore blowups, which only change the decomposable cycles (by the projective bundle formula). Looking at $\widetilde{\cC_s}^{\times k}$ ($k=2$ or $3$), there are hyperelliptic involutions $\sigma_i$ on the factors, with quotients $\mathscr{P}_i$ permutations of $\mathbb{P}^1 \times \cC_s^{\times (k-1)}$ and fixed points containing $\mathscr{Q}_i =$ a permutation of $\{ \tilde{q}_1,\tilde{q}_2\}\times \cC_s^{\times (k-1)}$.  We may of course decompose $\mathrm{CH}^a(\widetilde{\cC_s}^{\times k},b) = \sum_{\chi} \mathrm{CH}^a (\widetilde{\cC_s}^{\times k},b)^{\chi}$ according to the character thorugh which $\mathbb{Z}_2^{\times k}$ acts. In fact, writing
$$
Z=\sum_{\chi}\tfrac{1}{2^k}\sum_{\sigma\in \mathbb{Z}^k_2}\chi(\sigma)\sigma^* Z = \sum_{\chi} Z^{\chi} ,
$$
we can do this on the level of \emph{cycles}. If $\chi(\sigma_i) = -1$, then $Z^{\chi}$ pulls back to zero on $\mathscr{Q}_i$; while if $\chi(\sigma_i) = +1$, $Z^{\chi}$ is pulled back from $\mathscr{P}_i$. From this, one deduces that the image of $\delta$ merely equates cycles on each pair of components, leaving us with 3 copies of $\mathrm{CH}^2_{\text{ind}}(\widetilde{\cC_s}\times \widetilde{\cC_s},1) = \mathrm{CH}^2_{\text{ind}}(\widetilde{\cC_s}\times \widetilde{\cC_s},1)^{\chi_{12}}$. Here $\chi_{12}(\sigma_i)=-1$ ($i=1,2$), since pullbacks from $\widetilde{\cC_s}\times \mathbb{P}^1$ or $\mathbb{P}^1 \times \mathbb{P}^1$ are decomposable. Since this $\chi_{12}$-part restricts to zero on $\{\tilde{q}_j\}\times \widetilde{\cC_s}$ and $\widetilde{\cC_s}\times \{\tilde{q}_j\}$, it already lies in $\ker(d_1)\cap\ker(d_2)$. Taking $\mathscr{S}_3$-invariants gives (a). The same proof applies verbatim for (b).
\end{proof}

\begin{prop} \label{prop3.5}
The regulator of $\imath_{\cX}^* \tilde{\Delta}$ is nonzero in the Jacobian of Lemma \ref{lem3.5}(b) (which implies it is nonzero also in (a)).
\end{prop}

\begin{proof}
Follows at once from the commutative diagram
$$
\xymatrix{
\left( \mathrm{Gr}^{\mathscr{W}}_{-1} \chc^2 (\cX_s) \right)^{\mathscr{S}_3}_{\text{ind}} \ar [d]^{\overline{\mathrm{AJ}}} &&& \mathrm{CH}^2_{\text{ind}} (\tilde{A},1) \ar [d]^{\overline{\mathrm{AJ}}} \ar [lll]^{(\pi_s^{(1)})^*}
\\
J\left(\left(\mathrm{Gr}^W_2 H^3(\cX_s)\right)^{\mathscr{S}_3}_{\text{tr}}  (2) \right) & && J\left( H^2_{\text{tr}}(\tilde{A})(2) \right) \ar [lll]^{(\pi_s^{(1)})^*}
}
$$
and the fact that $\overline{\mathrm{AJ}} (\imath_{\cJ}^* Z_{\cC}) \neq 0$ on the right-hand side.
\end{proof}

%%%%%%%%%%%%%%%%%%%%%%%%%%%%%%%%%%%%%%%
\subsection{The normal function} \label{S3.6}
%%%%%%%%%%%%%%%%%%%%%%%%%%%%%%%%%%%%%%%

We assume that $\cC$ extends to a family $\cC^{\text{an}}$ over an analytic disk $\mathbb{D}$ (with $s$ its central point);\footnote{That is, $\cC^{\text{an}}\to\mathbb{D}$ resp. $\cC\to B$ are analytic resp. algebraic localizations of a family of genus 2 curves over a complex algebraic curve.} this is necessary in order to consider the normal functions associated to $Z_{\cC}^{\text{an}}$ and $\Delta^{\text{an}}$, which are sections of a family of nonalgebraic complex tori. We will drop the ``an'' in what follows.  Write $t$ for the coordinate on $\mathbb{D}$ (with $t(s)=0$).

Let $\cV$ denote the VHS over $\mathbb{D}^* = \mathbb{D}\setminus \{0\}$ associated to $\{ H^3(X_t)\}_{t\in \mathbb{D}^*}$, $\cV_{\text{alg}}$ the maximal level-one sub-VHS, and $\cV_{\text{tr}}$ the quotient.  Write $\cW_{\cdots}$ for the corresponding objects for $H^3(J(C_t))$, so that $\cW_{\text{tr}} \underset{\pi^*}{\hookrightarrow} \cV_{\text{tr}}$ with image the $\mathscr{S}_2$-invariants. Denote the normal functions by
$$
\nu_{Z_{\cC}} \in \mathrm{ANF}(\mathbb{D}^*,\cW(2))\text{ and }\nu_{\Delta},\nu_{\tilde{\Delta}} \in \mathrm{ANF}(\mathbb{D}^*,\cV(2))
$$
where $\nu_{\tilde{\Delta}} = \pi^* \nu_{Z_{\cC}}$. These are the sections of $J(\cW(2))$ resp. $J(\cV(2))$ obtained via fiberwise $\mathrm{AJ}$ of the cycles.  We write $\bar{\nu}$ for the projections to $\mathrm{ANF}(\mathbb{D}^*,\cW_{\text{tr}}^{(2)})$ resp. $\mathrm{ANF}(\mathbb{D}^*,\cV^{(2)}_{\text{tr}})$; these record fiberwise $\overline{\mathrm{AJ}}$ of the class of the cycles in the Griffiths group $\mathrm{Griff}^2(X_t)$ resp. $\mathrm{Griff}^2(J(\cC_t))$. But then since $\Delta_t \underset{\text{alg}}{\equiv} \tilde{\Delta}_t$, we have $\bar{\nu}_{\Delta} = \bar{\nu}_{\tilde{\Delta}}$.

Write $(\cdots)^N$ to denote $\ker(N)\subset (\cdots )$.  By \cite[Thm. II.B.9]{GGK}, we have a well-defined limit mapping
\begin{equation} \label{eq3(s)}
\mathrm{lim}_s: \mathrm{ANF}(\mathbb{D}^*,\cV(2)) \to J(\cV^N_{\text{lim}})(2)).
\end{equation}
Moreover, $\mathrm{lim}_s \nu_{\tilde{\Delta}}$ is given by $\mathfrak{r}^*\mathrm{AJ}(\imath_{\cX}^* \tilde{\Delta})$ in view of \cite[Thm. III.B.5]{GGK}, where $\mathfrak{r}^* : H^3(\cX_s)\to \cV_{\text{lim}}$ is the pullback via the Clemens retraction.  We need an extension of \eqref{eq3(s)} to $\cV_{\text{tr}}$. Consider the preimage $W_2^{\text{lim}} \mathrm{ANF}(\mathbb{D}^*,\cV(2))$ of $J\left( (W_2 \cV^N_{\text{lim}})(2) \right)$ under \eqref{eq3(s)}: its intersection $W_2^{\text{lim}} \mathrm{ANF}(\mathbb{D}^*,\cV_{\text{alg}}(2))$ has $\mathrm{lim}_s$ in $J\left( (W_2 \cV_{\text{alg,lim}}^N )(2) \right)$, and $W_2 \cV_{\text{alg,lim}}^N$ is of pure type $(1,1)$, hence dies in $(\mathrm{Gr}^W_2\cV_{\text{lim}}^N )_{\text{tr}}$.  So \eqref{eq3(s)} descends to a well-defined mapping
\begin{equation} \label{eq3(ss)}
\overline{\mathrm{lim}_s} : W_2^{\text{lim}} \mathrm{ANF} (\mathbb{D}^* ,\cV_{\text{tr}}(2)) \to J\left( (\mathrm{Gr}^W_2 \cV_{\text{lim}}^N )_{\text{tr}} (2) \right) .
\end{equation}
From \eqref{eq3(!!)} it is clear that $\nu_{\tilde{\Delta}}$ belongs to $W_2^{\text{lim}} \mathrm{ANF}(\mathbb{D}^* ,\cV(2))$ and so we may apply $\overline{\mathrm{lim}}_s$ to $\bar{\nu}_{\tilde{\Delta}}(=\bar{\nu}_{\Delta})$, to obtain
$$
\mathfrak{r}^* \overline{\mathrm{AJ}}(\imath^*_{\cX}\tilde{\Delta}) = \tfrac{1}{2} \mathfrak{r}^* (\pi_s^{(1)})^* \overline{\mathrm{AJ}}(\imath_{\cJ}^* Z_{\cC}) = \tfrac{1}{2} \mathfrak{r}^* (\pi^{(1)}_s)^* \mathrm{AJ} (Z_{\cC}^{(1)}).
$$
But $\mathrm{AJ}(Z_{\cC}^{(1)})\neq 0$ in the left-hand side of
\begin{equation} \label{eq412}
J\left(H^2_{\text{tr}}(A)(2)\right) \underset{(\pi^{(1)}_s)^*}{\hookrightarrow} J\left( \{ \mathrm{Gr}^W_2 H^3(\cX_s)\}_{\text{tr}}\right) \underset{\overline{\mathfrak{r}^*}}{\overset{\cong}{\rightarrow}} J\left( \{ \mathrm{Gr}^W_2 \cV^N_{\text{lim}}\}_{\text{tr}} (2) \right)
\end{equation}
and $\mathfrak{r}^*$ is an isomorphism on $W_2$. We conclude:

\begin{thm}\label{thm3.6}
Let $\bar{\nu}_{\Delta}$ be the section of $J(H^3_{\text{tr}}(X_t))$ over $\mathbb{D}^*$ associated to the Gross-Schoen cycle.  Then:
\begin{enumerate}[(i)]
\item $\overline{\mathrm{lim}_s} (\bar{\nu}_{\Delta})$ is nonzero, and given by $\mathrm{AJ}(Z_{\cC}^{(1)})$ via \eqref{eq412}, where $Z_{\cC}^{(1)} \in \mathrm{CH}^2_{\text{ind}}(A,1)$ is Collino's cycle; this implies that
\item $\bar{\nu}_{\Delta}$ is nonzero in $\mathrm{ANF}(\mathbb{D}^*,\cV_{\text{tr}}(2))$, and so
\item $\Delta$ is nontorsion in $\mathrm{Griff}^2(X_t)$ for very general $t$.
\end{enumerate}
\end{thm}

We have thus used regulator indecomposability of the specialization of the modified diagonal to check its generic algebraic inequivalence to zero.

%%%%%%%%%%%%%%%%%%%%%%%%%%%%%%%%%%%%%%%%%%%%%%%%%%%%%%%%%%%%%%%%%%%%%%%%%%%%%%%%%%%
\subsection{Second and third specializations of $Z_{\cC}$ and $\Delta$}

By adding a second parameter, we can allow $\cC_s$ to acquire an additional node $q'$, with normalization an elliptic curve $\tilde{E}$. Suppose moreover that he preimages $\{\tilde{q}_1 ',\tilde{q}_2 '\}$ (of $q'$) and $\{\tilde{q}_1 ,\tilde{q}_2 \}$ (of $q$) on $\tilde{E}$ are such that we have the equalities $\tilde{q}_2 ' - \tilde{q}_1 ' \equiv \tilde{q}_2 - \tilde{q}_1 \equiv 2(\tilde{q}_1 ' -\tilde{q}_1 ) =: \varepsilon$ of two-torsion classes. Then $A$ semistably degenerates to $E\times \bar{\mathbb{G}}_m$, where $E:= \tilde{E} / \langle \varepsilon \rangle$, and $Z^{(1)}_{\cC}$ specializes (goes up) to a class $Z_{\cC}^{(2)} \in \mathrm{CH}^2(E,2)$ which may be described as follows. Let $f,g\in \mathbb{C}(\tilde{E})^*$ have divisors $(f)=2[\tilde{q}_2 ] - 2[\tilde{q}_1 ]$ and $(g)=2[\tilde{q}_2 ' ] - 2[\tilde{q}_1 ' ]$, and satisfy $f(\tilde{q}_i ')=1$, $g(\tilde{q}_i ) = 1$ ($i=1,2$). Then the graph of the symbol $\{f,g \}$ belongs to $\mathrm{CH}^2 (\tilde{E},2)$, and $Z_{\cC}^{(2)}$ is the projection to $E$ of $\{f,g\} - \{f^{-1},g^{-1}\} \equiv 2\{f,g\}$. Its regulator can be shown to be nontorsion as in \cite[$\S$7]{Collino}, or by identifying $\{f,g\}$ as an Eisenstein symbol \cite[Example 10.1]{DoranKerr}.

Degenerating once more, in such a way that our four 4-torsion points ``remain finite'', $Z_{\cC}^{(2)}$ goes up to a cycle $Z_{\cC}^{(3)}\in \mathrm{CH}^2(\mathbb{C},3)$ given parametrically by ($z\mapsto$)
$$
\left( z, -\left(\tfrac{1-z}{1+z}\right)^2 , - \left( \tfrac{z-\mathbf{i}}{z+\mathbf{i}}\right)^2 \right) - \left(  z^{-1}, -\left(\tfrac{1-z}{1+z}\right)^{-2} , - \left( \tfrac{z-\mathbf{i}}{z+\mathbf{i}}\right)^{-2}\right) ,
$$
with regulator $32\mathbf{i} G$ (cf. \eqref{eqG}). This can be directly computed (as in $\S$\ref{toy})\footnote{This is done in \cite[$\S$IV.D]{GGK}, but with a small error as regards branches of log (which produces an extraneous term).} or done using two different formulas in \cite{DoranKerr} (cf. Example 10.1, and ``$D_5$'' in $\S$6.3). One may view this as a simple proof that $Z_{\cC}$, $Z_{\cC}^{(1)}$, and $Z_{\cC}^{(2)}$ are all nontorsion.

Here is an easy implication for the cycle $\Delta$ and its associated normal function, if we consider instead a good model for the triple fiber-product of the \emph{trinodal} degeneration of $C$. We get a specialization map from $\mathrm{ANF}(\mathbb{D}^*,\mathcal{V}_{\text{tr}}(2))$ to $\mathbb{C}/\mathbb{Q}(2)$ (along the lines of \cite[(IV.D.3)ff]{GGK}), under which $\bar{\nu}_{\Delta}$ goes to $16\mathbf{i}G$. This corresponds to specializing $\tilde{\Delta}$ to the special fiber of the good model, which is a complicated configuration of rational threefolds, with $\mathrm{Gr}^W_0 H^3$ of rank one.

%%%%%%%%%%%%%%%%%%%%%%%%%%%%%%%%%%%
%                                                  SECTION 5
%%%%%%%%%%%%%%%%%%%%%%%%%%%%%%%%%%%
\section{\bf Limits of higher normal functions}\label{MS5}
%%%%%%%%%%%%%%%%%%%%%%%%%%%%%%%%%%%
%%%%%%%%%%%%%%%%%%%%%%%%%%%%%%%%%%%

In this section we extend Proposition 6.2 of \cite{DoranKerr} to
the non-semistable setting, and provide a proof, which is omitted
in \cite{DoranKerr} for even the (semistable) case presented there.  We have found it more natural to work with motivic cohomology notation here; the reader who finds Chow cohomology notation more convenient may replace $H_{\cM}^a(X,\Q(b))$ by $\chc^b(X,2b-a)_{\Q}$.  All cycle groups in this section are taken to have $\Q$-coefficients.

\subsection{The Abel-Jacobi map for motivic cohomology of a normal crossing divisor\label{sec1.1}}

Let $\cX\overset{\bar{\pi}}{\to}\cS$ be a proper, dominant morphism
of smooth varieties, with unique singular fiber $\bar{\pi}^{-1}(0)=X_{0}=\cup Y_{i}$,
and $\dim\cX=d$, $\dim\cS=1$. Assume first that $X_{0}$ is a SNCD, so as to be able to make the descent spectral sequence for $H_{\cM}$ and $H_{\cD}$ explicit. To this end, we shall
write $Y_{I}:=\cap_{i\in I}Y_{i}$, $Y^{[\ell]}=\amalg_{|I|=\ell+1}Y_{I}$,
$\jmath_{I,j}:Y_{I\cup\{j\}}\hookrightarrow Y_{I}$, $Y^{I}:=\cup_{j\notin I}Y_{I\cup\{j\}}\subset Y_{I}$,
and $\langle i\rangle_{I}$ for the position of $i$ in $I$.

Recall (from \cite{KL,GGK}) that there are double complexes

\begin{align*}
Z_{Y}^{\ell,k}(p) &:=\oplus_{|I|=\ell+1}Z_{\#}^{p}(Y_{I},-k)\mspace{50mu}\text{resp.}\\
K_{Y}^{\ell,k}(p) &:=B_{Y}^{\ell,k}(p)\oplus F^{p}D_{Y}^{\ell,k}(p)\oplus D_{Y}^{\ell,k-1}(p)\\
&:=\oplus_{|I|=\ell+1}\left\{ C_{\#}^{2p+k}(Y_{I};\QQ(p))\oplus F^{p}D_{\#}^{2p+k}(Y_{I})\oplus D_{\#}^{2p+k-1}(Y_{I})\right\} ,
\end{align*}
with $d_{0} = \del_{\cB}$ (Bloch differential) resp. $D$ (cone differential
$D(\alpha,\beta,\gamma):=(-d\alpha,-d\beta,d\gamma-\beta+\delta_{\alpha}$))
and $d_{1}=\del_{\fI}=\sum_{|I|=\ell+1}\sum_{j\notin I}(-1)^{\langle j\rangle_{I\cup\{j\}}}(\jmath_{I,j})^{*}$,
whose associated
 simple complexes compute motivic resp. Deligne cohomology:
  \begin{align*}
H_{\cM}^{2p-r}(X_{0},\QQ(p)) &=H^{-r}(Z_{Y}^{\bullet}(p),\ul{\ul{\del_\cB}}),\\
H_{\cD}^{2p-r}(X_{0},\QQ(p)) &=H^{-r}(K_{Y}^{\bullet}(p),\ul{\ul{D}}).
\end{align*}
Briefly, $D_{\#}^{\bullet}(Y_{I}):=\cN^{\bullet}\{Y^{I}\}(Y_{I})$
denotes normal currents of intersection type, $C_{\#}^{\bullet}(Y_{I};\QQ(p)):= \mathcal{I}^{\bullet}\{Y^{I}\}(Y_{I})\otimes_{\ZZ}\QQ(p)$
the locally integral currents contained therein, and $Z_{\#}^{p}(Y_{I},\bullet):=Z_{\RR}^{p}(Y_{I},\bullet)_{Y^{I}}\subset Z^{p}(Y_{I},\bullet)$
the quasi-isomorphic subcomplex of higher Chow precycles in (real)
good position with respect to $Y_{I}$.  (For background on currents of intersection type, the reader may consult the treatments in \cite[$\S$III.A]{GGK} and \cite[$\S$8]{KL}.)

The KLM formula \cite{KLM,KL}, which takes the form 
\[
\cW\mapsto(-2\pi\ay)^{p+k}\left((2\pi\ay)^{-k}T_{\cW},\Omega_{\cW},R_{\cW}\right),
\]
provides a morphism of double complexes $Z_{Y}^{\ell,k}(p)\to K_{Y}^{\ell,k}(p)$
which induces the Abel-Jacobi map%
\footnote{We concentrate on the $r>0$ case since $r=0$ has been treated in
\cite{GGK}.%
}
\[
\mathrm{AJ}_{X_{0}}^{p,r}:H_{\cM}^{2p-r}(X_{0},\QQ(p))\to H_{\cD}^{2p-r}(X_{0},\QQ(p))\underset{r>0}{\cong}J^{p,r}(X_{0}).
\]
Given 
\[
Z=\{Z^{[\ell]}\}_{\ell\geq 0}
=\{Z_{I}^{[\ell]}\}_{\ell\geq0,|I|=\ell+1}\in\ker(\ul{\ul{\del_{\cB}}})\subset
\oplus_{\ell\geq 0}Z_{Y}^{\ell-\ell-r}(p)=Z_{Y}^{-r}(p),
\]
there exist $\Xi^{[\ell]}\in F^{p}D_{Y}^{-r-1}(p)$, $\Gamma^{[\ell]}\in B_{Y}^{-r-1}(p)$
such that\begin{multline*}
\left\{ (-2\pi\ay)^{p-\ell}\left((2\pi\ay)^{\ell}T_{Z^{[\ell]}},\Omega_{Z^{[\ell]}},R_{Z^{[\ell]}}\right)\right\} _{\ell\geq0}
\\
-\,\uud\,\left\{ (-2\pi\ay)^{p-\ell}\left((2\pi\ay)^{\ell}\Gamma^{[\ell]},\Xi^{[\ell]},0\right)\right\} _{\ell\geq0}
\end{multline*}
\[
=\;\left\{ (-2\pi\ay)^{p-\ell}\left(0,0,R_{Z^{[\ell]}}+\Xi^{[\ell]}-(2\pi\ay)^{\ell}\delta_{\Gamma^{[\ell]}}\right)\right\} _{\ell\geq0}\in\ker(\uud)\subset D_{Y}^{-r-1}(p)
\]
yields a class $\widetilde{\mathrm{AJ}(Z)}\in H^{2p-r-1}(X_{0},\CC)$ projecting
to
\[
\mathrm{AJ}(Z)\,\in\, J^{p,r}(X_{0})\,=\,\frac{H^{2p-r-1}(X_{0},\CC)}{F^{p}H^{2p-r-1}(X_{0},\CC)+H^{2p-r-1}(X_{0},\QQ(p))}
\]
\begin{equation}\label{eqp2sharp}\cong \, \left. \left\{ F^{-p+1} H_{2p-r-1}(X_0,\CC)\right\}^{\vee} \right/ H^{2p-r-1}(X_0,\QQ(p)) .
\end{equation}Now consider the double complex 
\[
[F^{q}]\, D_{\ell,k}^{Y}(-p)\,:=\,\oplus_{|I|=\ell+1}[F^{d+q-\ell-1}]\, D^{2(d-p-\ell)-k-1}(Y_{I})  
\]
with $d_{0}=d$, $d_{1}=\text{Gy}:=2\pi\ay\sum_{|I|=\ell}(-1)^{\langle i\rangle_{I}}(\jmath_{I\backslash\{i\},i})_{*}$,
which computes homology:  

 \begin{equation}\label{eqp3*}H_{-r}\left( F^{-p+1} D_{\bullet}^Y (-p) \right) = F^{-p+1} H_{2p-r-1}(X_0,\CC) .
\end{equation} By \cite[Prop. III.A.13]{GGK}, \eqref{eqp3*} can be represented
by elements of the form 
\[
\omega=\{\omega^{[\ell]}\}_{\ell\geq0}\subset\oplus_{\ell\geq0}F^{d-p-\ell}A^{2(d-p)-\ell+r-1}(Y^{[\ell]})\langle\log(\cup_{|I|=\ell+1}Y^{I})\rangle,
\]
and then $(-2\pi\ay)^{r-p}\langle\widetilde{\mathrm{AJ}(Z)},\omega\rangle=$\begin{equation}\label{eqp3dagger}\sum_{\ell\geq 0}\left( \int_{Y^{[\ell]}} R_{Z^{[\ell]}} \wedge \omega^{[\ell]} - (2\pi \ay)^{r+\ell} \int_{\Gamma^{[\ell]}} \omega^{[\ell]} \right).
\end{equation}The integrals here converge by \cite[Lemma III.A.6]{GGK}. In the
event that
\[
Z=\{Z_{i}\}\in Z_{\#}^{p}(X_{0},r):=\ker(\ul{\ul{\del_{\cB}}})\cap Z_{Y}^{0,-r}(p),
\]
we can arrange to have $\Gamma^{[\ell]}=0$ $\forall\ell>0$ \cite[III.A.19]{GGK},
reducing \eqref{eqp3dagger} to\begin{equation}\label{eqp4ddagger}(-2\pi\ay)^{r-p} \langle \widetilde{\mathrm{AJ}(Z)} ,\omega \rangle = \sum_i \left( \int_{Y_i} R_{Z_i}\wedge \omega_i - (2\pi\ay)^r \int_{\Gamma_i} \omega_i \right) .
\end{equation}

\subsection{Limits of higher normal functions in the semistable setting\label{sec1.2}}

Turning to normal functions, we begin with the morphisms
\[
\cX^{*}:=\cX\backslash X_{0}\overset{\pi}{\to}\cS\backslash\{0\}=:\cS^{*}\overset{\jmath}{\hookrightarrow}\cS,
\]
and write $\VV=R^{2p-r-1}\pi_{*}\QQ(p)$, $\cV$ (resp. $\cV_{e}$)
for the corresponding weight-$(-r-1)$ VHS (resp. its canonical extension).
Below we will abuse notation by writing $\cV$ (resp. $\cV_{e}$)
also for its sheaf of sections $\VV\otimes\cO_{\Delta^{*}}$ (resp.
$\tilde{\VV}\otimes\cO_{\Delta}$) for a disk $\Delta\subset\cS$
about $\{0\}$. Denote the LMHS at $\{0\}$ by $\left(\tilde{\VV}_{0},F_{lim}^{\bullet},W_{\bullet}\right)=V_{lim}$,
with the monodromy logarithm $N=\log(T)$ and $\tilde{\VV}=e^{-\frac{\log(s)}{2\pi\ay}N}\VV$.  (A general reference for 
the canonical extension and degenerations of Hodge structure may be found in \cite[$\S$11.1]{PS}.)

By a higher normal function $\nu\in ANF_{\cS^{*}}^{r}(\cV)$, we shall
mean an admissible VMHS of the form\begin{equation}\label{eqp5*}0 \to\cV \to \cE_{\nu} \to \QQ_{\cS^*} (0) \to 0 \, ;
\end{equation}the action of $N$ extends to the underlying local system $\mathbb{E}_{\nu}$
(which yields $\tilde{\EE}_{\nu}$, $E_{\nu}^{lim}$, $\mathcal{E}_{\nu,e}$).
Write $\imath_0 :\{0\}\hookrightarrow \mathcal{S}$ and $\imath:X_0 \hookrightarrow \mathcal{X}$. Applying the composition 
\[
\mathrm{AVMHS}(\cS^{*})\overset{R\jmath_{*}}{\to}D^{b}\mathrm{MHM}(\cS)\overset{\imath_{0}^{*}}{\to}D^{b}\mathrm{MHS}
\]
of exact functors to $\cV$ yields (up to quasi-isomorphism) the complex
$\cK^{\bullet}:=\left\{ V_{0}^{lim}\overset{N}{\to}V_{0}^{lim}(-1)\right\} $.
Therefore, applying it to \eqref{eqp5*} yields a diagram \[
\xymatrix{
& 0
\\
& \text{Hom}_{\mathrm{MHS}}\left( \QQ(0),H^1 \cK^{\bullet} \right) \ar [u]
\\
\mathrm{ANF}_{\cS^*}^r (\cV) \ar [ur]^{\mathrm{sing}_0 \mspace{30mu}} \ar [r]_{\imath_0^* R\jmath_* \mspace{70mu}} & \text{Ext}^1_{D^b \mathrm{MHS}} \left( \QQ(0),\cK^{\bullet}\right) \ar [u]
\\
\ker (\mathrm{sing}_0) \ar @{^(->} [u] \ar [r]_{\mathrm{lim}_0 \mspace{70mu}} & \text{Ext}^1_{\mathrm{MHS}} \left( \QQ(0),H^0\cK^{\bullet}\right) \ar [u] \; .
\\
& 0 \ar[u]
}
\]
defining the invariants $\mathrm{sing}_{0}$ and $\mathrm{lim}_{0}$.
Of course, we may also view $\nu$ as a (horizontal, holomorphic)
section of the Jacobian bundle $\cJ(\cV)=\cV/(\cF^{0}\oplus\VV)$
over $\cS^{*}$, by taking the difference of lifts $\nu_{\QQ}(s)\in\EE_{\nu,s}$
resp. $\nu_{F}(s)\in F^{0}E_{\nu,s}$ of $1\in\QQ(0)$ in $V_{s}$.
In this context, admissibility means that we also have (for some disk%
\footnote{we will freely shrink this as needed without further comment%
} $\Delta\subset\cS$ about $0$):
\begin{lyxlist}{00.00.0000}
\item [{(a)}] a lift $\nu_{F}$ of $1\in\QQ_{\cS}(0)$ to $\cE_{\nu,e}$
with $\nu_{F}|_{\Delta^{*}}$ in $\cF^{0}\cE_{\nu}$; and
\item [{(b)}] a lift $\nu_{\QQ}$ of $1$ to $\tilde{\EE}_{\nu,0}$ satisfying
$N\nu_{\QQ}\in W_{-2}\tilde{\VV}_{0}$.
\end{lyxlist}
One then has 
\[
\mathrm{sing}_{0}(\nu)=[N\nu_{\QQ}]\in\text{Hom}_{\mathrm{MHS}}\left(\QQ(0),\frac{V_{0}^{lim}}{NV_{0}^{lim}}(-1)\right)\cong H^{1}(\Delta^{*},\VV)^{(0,0)}.
\]
If this vanishes, then $\nu_{\QQ}$ may be chosen in $\ker(N)$, so
that $\tilde{\nu}:=\nu_{\QQ}-\nu_{F}$ gives a well-defined section
of $\cV_{e}$ (over $\Delta$) with $\nabla(\nu_{\QQ}-\nu_{F})|_{\Delta^{*}}\in\Gamma(\Delta^{*},\cF^{-1}\cV)$
by horizontality. Using $\mathrm{Res}_{0}(\nabla)=-2\pi\ay N$, we
find that 
\[
\widetilde{\mathrm{lim}_{0}\nu}:=\tilde{\nu}(0)=\nu_{\QQ}-\nu_{F}(0)\in\ker\left\{ V_{0}^{lim}\overset{N}{\to}\frac{V_{0}^{lim}}{F^{-1}}\right\} =\ker(N)+F^{0}V_{0}^{lim}
\]
which projects to compute $\mathrm{lim}_{0}\nu\in\text{Ext}_{\mathrm{MHS}}^{1}(\QQ(0),\ker(N))$.

By \cite[III.B.7]{GGK}, a holomorphic section $\omega(s)\in\Gamma\left(\cS,(\cF^{1}\cV^{\vee})_{e}\right)$
of the canonical extension may be represented by a $d_{rel}$-closed
$C^{\infty}$ relative $\log\langle X_{0}\rangle$ $\left(2(d-p)+r-1\right)$-form
on $\cX_{\Delta}$, and we write $\mathrm{lim}_{0}\omega$ for its
restriction to $(\cF^{1}\cV^{\vee})_{e,0}$. Referring to the (dual)
portions \begin{equation*}
\to H^{2p-r-1}(X_{0})(p)\overset{\mathfrak{r}^{*}}{\to}H_{lim}^{2p-r-1}(X_{t})(p)
\overset{N}{\to}H_{lim}^{2p-r-1}(X_{t})(p-1)\to,
\end{equation*}and\begin{multline*}
\to H^{2(d-p)+r-1}_{lim}(X_{t})(d-p)\overset{N}{\to} H_{lim}^{2(d-p)+r-1}(X_{t})(d-p-1)\\ \overset{\mathfrak{r}_{*}}{\to}
H_{2p-r-1}(X_{0})(-p)\to
\end{multline*}of the Clemens-Schmid sequence, the pullbacks $\omega_{i}$ (and their
iterated residues $\omega_{I}$ on substrata) define a representative
(as described after \eqref{eqp3*}) of 
\[
\mathfrak{r}_{*}(\mathrm{lim}_{0}\omega)=:\omega(0)\in F^{-p+1}H_{2p-r-1}(X_{0},\CC).
\]
Note that $\left\langle \widetilde{\mathrm{lim}_{0}\nu},\mathrm{lim}_{0}\omega\right\rangle =\lim_{s\to0}\langle\tilde{\nu}(s),\omega(s)\rangle\in\CC.$

To construct a normal function with $\mathrm{sing}_{0}(\nu)=0$, let
$\mathfrak{z}\in\ker(\del\cB)\subset Z_{\#}^{p}(\cX,r)$ be a representative
of a class $\Xi\in \CH^{p}(\cX,r)$ meeting all $Y_{I}$ properly, and
define $\mathfrak{z}_{0}=\{Z_{i}\}\in Z_{\#}^{p}(X_{0},r)$ by $Z_{i}:=\mathfrak{z}\cdot Y_{i}$.
This represents $\imath^{*}\Xi\in H_{\cM}^{2p-r}(X_{0},\QQ(p))$,
where $\imath:X_{0}\hookrightarrow\cX$. In a neighborhood $\cX_{\Delta}:=\bar{\pi}^{-1}(\Delta)$
of $X_{0}$, $\mathfrak{z}$ (hence $T_{\mathfrak{z}}$) meets all
fibers properly, and (since $H^{2p-r}(\cX_{\Delta})\cong H^{2p-r}(X_{0})$)
we may choose an integral current $\tilde{\Gamma}$ on $\cX_{\Delta}$
with $\del\tilde{\Gamma}=T_{\mathfrak{z}}$ meeting the $Y_{i}$ and
all fibers properly. Clearly then $\tilde{R}_{\mathfrak{z}}:=R_{\mathfrak{z}}-(2\pi\ay)^{r}\delta_{\tilde{\Gamma}}$
is a closed current on $\cX_{\Delta}$, of intersection type with
respect to the $Y_{I}$. Setting $\Gamma_{i}:=\tilde{\Gamma}\cdot Y_{i}$,
we have by \eqref{eqp4ddagger} that the restriction of $\tilde{R}_{\mathfrak{z}}$
to the $Y_{i}$ computes a lift to $H^{2p-r-1}(X_{0},\CC)$ of $\mathrm{AJ}_{X_{0}}(\imath^{*}\Xi)$.
Moreover, over $\Delta^{*}$ the normal function $\nu(s)=\mathrm{AJ}_{X_{s}}(\Xi_{s})$
associated to $\Xi^{*}\in \CH^{p}(\cX^{*},r)$ is computed by the fiberwise
restrictions
\[
\left[\tilde{R}_{\mathfrak{z}}|_{X_{s}}\right]\in H^{2p-r-1}(X_{s},\CC)\twoheadrightarrow J^{p,r}(X_{s})\cong\frac{\left\{ F^{d-p}H^{2(d-p)+r-1}(X_{s},\CC)\right\} ^{\vee}}{\text{periods}}.
\]
Putting everything together, we have \begin{align*}
\langle \widetilde{\mathrm{lim}_0 \nu} , \mathrm{lim}_0 \omega \rangle &= \lim_{s \to 0} \int_{X_s}\tilde{R}_{\mathfrak{z}} |_{X_s} \wedge \omega(s) 
\\
&= \sum_i \int_{Y_i} \tilde{R}_{\mathfrak{z}}|_{Y_i} \wedge \omega_i
\\
&= \langle \widetilde{\mathrm{AJ}_{X_0} (\imath^* \Xi )} ,\omega(0)\rangle
\\
&= \langle \mathfrak{r}^* \widetilde{\mathrm{AJ}_{X_0} (\imath^* \Xi )},\mathrm{lim}_0 \omega \rangle \, .
\end{align*}The second equality is the crucial one; it comes about by noting that
$\tilde{R}_{\mathfrak{z}}\wedge\omega$ is of $X_{0}$-intersection
type, hence the $0$-current $\left(\bar{\pi}|_{\cX_{\Delta}}\right)_{*}\left(\tilde{R}_{\mathfrak{z}}\wedge\omega\right)$
is of $\{0\}$-intersection type. Since it is also holomorphic on
$\Delta^{*}$, it follows that it is holomorphic (hence continuous)
on $\Delta$. So we have proved that\begin{equation}\label{ssdThm}\mathrm{lim}_{{s \to}0}{\mathrm{AJ}}_{X_{s}}(\Xi_{s})=J(\mathfrak{r}^{*}){\mathrm{AJ}}_{X_{0}}(\imath^{*}\Xi),
\end{equation}for $\mathfrak{z}$ as above and $\cX_{\Delta}\to\Delta$ semistable.
\begin{rem}
It is the SSD case which most clearly exhibits the phenomenon of ``going
up in $K$-theory in the limit''. Recall from $\S$\ref{S3ex} that the semi-simplicial structure of $X_{0}$
gives rise to a ``weight'' filtration $\mathscr{W}_{\bullet}$ on $H_{\cM}^{2p-r}(X_{0},\QQ(p))$,
with $\mathscr{W}_{-b}$ consisting of the classes which admit a representative
in $\oplus_{\ell\geq b}Z_{Y}^{\ell,-\ell-r}(p)$, and $Gr_{-b}^{\mathscr{W}}$
a subquotient of $\CH^{p}(Y^{[b]},r+b)$. So the degree of $K$-theory
``goes up'' if $\imath^{*}\Xi\in\mathscr{W}_{-b}$ for $b>0$.
\end{rem}

\subsection{Limits in the general setting\label{sec1.3}}

To state the more general result, we now drop the SSD assumption on
$\bar{\pi}$, hence the assumption of unipotency of $\VV$ at $\{0\}$
(i.e. of $T$). One still has pullback and $\mathrm{AJ}$ maps
\[
\CH^{p}(\cX,r)\overset{\imath^{*}}{\to}H_{\cM}^{2p-r}(X_{0},\QQ(p))\overset{{\mathrm{AJ}}_{X_{0}}^{p,r}}{\to}J^{p,r}(X_{0}),
\]
where $J^{p,r}(X_{0}):=\text{Ext}_{{\rm MHS}}^{1}(\QQ(0),H^{2p-r-1}(X_{0},\QQ(p)))$.
Write $T=T_{ss}T_{un}$ for the Jordan decomposition, $\kappa$ for
the order of $T_{ss}$, $s$ for the coordinate on $\Delta$, and
$N:=\log T_{un}$. Note that $\ker(N)\,(=\ker(T^{\kappa}-I)\,)\supsetneq\ker(T-I)$,
unless $\kappa=1$. The portions of Clemens-Schmid
\[
\to H_{2d-2p+r+1}(X_{0})(-d)\overset{\imath^{*}\imath_{*}}{\to}H^{2p-r-1}(X_{0})\overset{\mathfrak{r}^{*}}{\to}H_{lim}^{2p-r-1}(X_{s})
\]
and
\[
H_{lim}^{2p-r-1}(X_{s})(-1)\overset{\mathfrak{r}_{*}}{\to}H_{2d-2p+r-1}(X_{0})(-d)\overset{\imath^{*}\imath_{*}}{\to}H^{2p-r+1}(X_{0})\to
\]
remain exact sequences of MHS, with $\mathrm{im}(\mathfrak{r}^{*})=H_{lim}^{2p-r-1}(X_{s}):=\ker(T-I)\subseteq H_{lim}^{2p-r-1}(X_{s})$.
(This is a sub-MHS although $T-I$ itself is not a morphism of MHS
from $H_{lim}^{2p-r-1}$ to $H_{lim}^{2p-r-1}(-1)$.) As above, we
write $\cV$ for the VHS and $\cJ(\cV)$ for the family of generalized
intermediate Jacobians.

Let $\sigma:\,\hat{\cS}^{*}\to\cS^{*}$ be a cyclic cover extending
the map $t\mapsto t^{\kappa}(=s)$ from $\Delta^{*}\to\Delta^{*}$,
with $\mu\in\text{Aut}(\hat{\cS}^{*}/\cS^{*})$ a generator, and $\hat{\cV}$
resp. $\hat{\VV}$ the (unipotent) pullback variation resp. local
system. We have the canonical extension $\cJ(\hat{\cV}_{e}):=\hat{\cV}_{e}/\{\hat{\cF}_{e}^{0}+\jmath_{*}\hat{\VV}\}$,
with fiber over $\{0\}$ $\cJ_{lim}^{p,r}:=\hat{\cV}_{e,0}/\{(\jmath_{*}\hat{\VV})_{0}+\hat{\cF}_{e,0}^{0}\}$,
and write 
\[
J(\mathfrak{r}^{*}):\, J^{p,r}(X_{0})\to J_{lim}^{p,r}
\]
for the map induced by $\mathfrak{r}$, with image $J_{inv}^{p,r}:=\text{Ext}_{\mathrm{MHS}}^{1}(\QQ(0),\ker(T-I))$.
Moreover, there is a diagram\[ 
\xymatrix{
X_0 \ar @{^(->} [d]^{\imath} & \hat{X}_0 ' \ar @{^(->} [d]^{\hat{\imath} '} \ar @{->>} [l]^{P_0} \ar @{->>} [r]_{Q_0} & \hat{X}_0 \ar @{^(->} [d]^{\hat{\imath}}
\\
\cX \ar [d]^{\pi} & \hat{\cX} ' \ar [d]^{\hat{\pi} '} \ar @{->>} [r]_{\bar{Q}} \ar @{->>} [l]^{\bar{P}} & \hat{\cX} \ar [d]^{\hat{\pi}}
\\
\cS & \hat{\cS} \ar @{->>} [l]^{\bar{\sigma}} \ar @{=} [r] & \hat{\cS} 
}
\]with $\hat{\cX}_{\Delta}:=\hat{\pi}^{-1}(\Delta)\to\Delta$ semistable,
$\hat{\cX}'$, $\hat{\cX}$ smooth, and $\hat{\cX}'\backslash\hat{\cX}_{0}'=\hat{\cX}\backslash\hat{X}_{0}=\hat{\cX}_{0}^{*}:=\cX^{*}\times_{\sigma}\hat{\cS}.$
(That is, $\bar{Q}$ restricts to the identity on $\hat{\cX}^{*}$;
write $P$ for the restriction of $\bar{P}$ to $\hat{\cX}^{*}\to\cX^{*}$.)
Note that we have $H_{lim}^{*}(\hat{X}_{t})\cong H_{lim}^{*}(X_{s})$.
The natural map
\[
J(\hat{\mathfrak{r}}^{*}):\, J^{p,r}(\hat{X}_{0})\to J_{lim}^{p,r}
\]
has image $\hat{J}_{inv}^{p,r}:=\text{Ext}_{\mathrm{MHS}}^{1}(\QQ(0),\ker(T^{\kappa}-I))$.

By definition of admissibility, we have a pullback map\begin{align*}
\sigma^* : \, & \text{ANF}_{\cS^*}(\cV)  \to  \text{ANF}_{\hat{\cS}^*}(\hat{\cV}) \\ & \mspace{60mu} \nu \; \longmapsto \; \hat{\nu}\;\; ,
\end{align*}and if $\text{sing}_{0}(\nu):=\text{sing}_{0}(\hat{\nu})=0$, we define
$\text{lim}_{0}\nu:=\text{lim}_{0}\hat{\nu}\in\hat{J}_{inv}^{p,r}$.
The following result extends Proposition 6.2 of \cite{DoranKerr}:
\begin{thm}
\label{limThm}Let $\Xi^{*}\in \CH^{p}(\cX^{*},r)$ $(r>0)$ be given,
with 
\[
cl^{p,r}(\Xi^{*})\in\mathrm{Hom}_{\mathrm{MHS}}(\QQ(0),H^{2p-r}(\cX^{*},\QQ(p)))
\]
and 
\[
\nu_{\Xi^{*}}(s):=\mathrm{AJ}_{X_{s}}(\Xi_{s})\in\mathrm{ANF}_{\cS^{*}}(\cV),
\]
where $\Xi_{s}:=\imath_{X_{s}}^{*}(\Xi^{*})$.

(a) Suppose 
\[
Res_{X_{0}}\left(cl^{p,r}(\Xi^{*})\right)=0\in\mathrm{Hom}_{\mathrm{MHS}}\left(\QQ(0),H_{2(d-p)+r-1}(X_{0},\QQ(-d))\right).
\]
Then $\mathrm{sing}_{0}(\nu_{\Xi^{*}})=0$, and $\mathrm{lim}_{0}(\nu_{\Xi^{*}})$
lies in $J_{inv}^{p,r}$.

(b) If $\Xi^{*}$ is the restriction of $\Xi\in \CH^{p}(\cX,r)$, then
we have 
\[
\mathrm{lim}_{0}(\nu)=J(\mathfrak{r}^{*})\left(\mathrm{AJ}_{X_{0}}(\imath^{*}\Xi)\right).
\]
\end{thm}
\begin{proof}
$(a)$ Set $\hat{\Xi}^{*}:=P^{*}(\Xi^{*})$. The assumption implies
that $cl^{p,r}(\Xi^{*})$ lifts to 
\[
\xi\in\text{Hom}_{\mathrm{MHS}}(\QQ(0),H^{2p-r}(\cX,\QQ(p))),
\]
and then $cl^{p,r}(\hat{\Xi}^{*})$ lifts to $\bar{Q}_{*}\bar{P}^{*}\xi$,
hence has trivial $Res_{\hat{X}_{0}}$. It follows at once that $(\text{sing}_{0}(\nu_{\Xi^{*}})=)\,\text{sing}_{0}(\nu_{\hat{\Xi}^{*}})=0$, in view of the diagram\[
\xymatrix{
\mathrm{ANF}_{\hat{\cS}^*}(\hat{\cV}) \ar @/^2pc/ [rr]^{\text{sing}_0} \ar [r]^{[\cdot ]} & H^1(\hat{\cS}^* , \hat{\VV}) \ar [r]^{|_{\Delta}} \ar [d] & H^1(\Delta^* ,\hat{\VV}) \ar @{^(->} [rd] \ar [d]
\\
\CH^p(\hat{\cX}^*,r) \ar [u] \ar [r]^{cl^{p,r}} & H^{2p-r}(\hat{\cX}^*) \ar [r]^{|_{\Delta}} & H^{2p-r}(\hat{\cX}^*_{\Delta}) \ar [r]^{Res} & H^{2p-r+1}_{\hat{X}_0}(\hat{\cX}) \, .
}
\]Using admissibility, $\nu_{\hat{\Xi}^{*}}$ lifts to a section of
$\cJ(\hat{\cV}_{e})$ with value $\text{lim}_{0}(\nu_{\hat{\Xi}^{*}})\in\hat{J}_{inv}^{p,r}$
at $0$.

Now $\mu$ lifts to $\cM\in\text{Aut}(\hat{\cX}^{*}/\cX^{*})$, which
evidently acts on $(\jmath_{*}\hat{\VV})_{0}$ as an automorphism
of MHS. That is, the \emph{restriction} of $T$ to $\ker(T^{\kappa}-I)\subset H_{lim}^{2p-r-1}$
is MHS-compatible, and so $T$ acts on $\hat{J}_{inv}^{p,r}$ with
fixed locus $J_{inv}^{p,r}$. Since $\nu_{\hat{\Xi}^{*}}=\sigma^{*}\nu_{\Xi^{*}}$,
we have $\nu_{\hat{\Xi}^{*}}=\mu^{*}\nu_{\hat{\Xi}^{*}}$ and taking
$\text{lim}_{0}$ on both fibers gives $\text{lim}_{0}(\nu_{\hat{\Xi}^{*}})=T\,\text{lim}_{0}(\nu_{\hat{\Xi}^{*}})$.

$(b)$ Write $\hat{\Xi}':=\bar{P}^{*}\Xi$, $\hat{\Xi}:=\bar{Q}_{*}\hat{\Xi}'$,
$\hat{\Xi}'':=\bar{Q}^{*}\hat{\Xi}$, $\Xi_{0}:=\imath^{*}\Xi$, $\hat{\Xi}_{0}':=(\hat{\imath}')^{*}\hat{\Xi}'$,
etc.; note that $P_{0}^{*}\Xi_{0}=\hat{\Xi}_{0}'$, $Q_{0}^{*}\hat{\Xi}_{0}=\hat{\Xi}_{0}''$,
and $\hat{\Xi}'-\hat{\Xi}''=\hat{\imath}_{*}'\xi_{0}$ for some $\xi_{0}\in \CH^{p-1}(\hat{X}_{0}',r)$.
We have the motivic homology $\mathrm{AJ}$ map $\mathrm{AJ}^{\hat{X}_{0}'}:\, \CH^{p-1}(\hat{X}_{0}',r)\to\text{Hom}_{\mathrm{MHS}}(\QQ(0),H_{2(d-p)+r+1}(X_{0},\QQ(-d)))$,
and using functoriality of $\mathrm{AJ}$\begin{align*}
J(P_0^*)\left(\mathrm{AJ}_{X_0}(\Xi_0)\right) &=\mathrm{AJ}_{\hat{X}_0}(\hat{\Xi}_0 ')
\\
&= \mathrm{AJ}_{\hat{X}_0}(\hat{\Xi}_0 '') + \mathrm{AJ}_{\hat{X}_0} \left( (\hat{\imath} ')^* \hat{\imath}_* ' \xi_0 \right) 
\\
&= J(Q_0^*)\left( \mathrm{AJ}_{\hat{X}_0} (\hat{\Xi}_0)\right) + J\left( (\hat{\imath}')^* \hat{\imath}_* ' \right) \mathrm{AJ}^{\hat{X}_0 '}(\xi_0).
\end{align*}Since $(\hat{\mathfrak{r}}')^{*}\circ(\hat{\imath}')^{*}\hat{\imath}_{*}'=0$,
$P_{0}\circ\hat{\mathfrak{r}}'=\mathfrak{r}$, and $Q_{0}\circ\hat{\mathfrak{r}}'=\hat{\mathfrak{r}}$,
applying $J((\hat{\mathfrak{r}}')^{*})$ and using \eqref{ssdThm}
gives
\[
J(\mathfrak{r}^{*})\left(\mathrm{AJ}_{X_{0}}(\Xi_{0})\right)=J(\hat{\mathfrak{r}}^{*})\left(\mathrm{AJ}_{\hat{X}_{0}}(\hat{\Xi}_{0})\right)=\text{lim}_{0}(\nu_{\hat{\Xi}^{*}})=\text{lim}_{0}(\nu_{\Xi^{*}}).
\]

\end{proof}

\begin{rmk}
A similar result holds for $r=0$; details are left to the reader.
\end{rmk}

\subsection{Limits of truncated normal functions\label{sec1.4}}

Continuing in the setting of \S\ref{sec1.3}, recall that the fiber
over $\{0\}$ of the canonical extension $(\cV^{\vee})_{e}$ decomposes
as a direct sum of generalized eigenspaces ${\bf E}_{\lambda}$ for
$Res_{s=0}(\nabla)$, with eigenvalues in $[0,1)$. The natural morphism
$\bar{\sigma}^{*}(\cV^{\vee})_{e}\overset{\rho}{\to}\hat{\cV}_{e}^{\vee}$
has kernel the skyscraper sheaf $\oplus_{\lambda\in(0,1)}{\bf E}_{\lambda}$
over $\{0\}$. We may use the composition\begin{multline*}
\Gamma\left( \Delta , (\cV^{\vee})_e \right) \overset{\mathrm{lim}_0}{\to} (\cV^{\vee})_{e,0} \overset{\rho|_0}{\to} \hat{\cV}^{\vee}_{e,0} \cong H_{lim}^{2(d-p)+r-1}(\hat{X}_t)(d-p-1)
\\
\overset{\hat{\mathfrak{r}}_* '}{\to} H_{2p-r-1}(\hat{X}_0 ')(-p) \overset{(P_0)_*}{\to} H_{2p-r-1}(X_0)(-p)
\end{multline*}to define $\omega(0)\in H_{2p-r-1}(X_{0},\CC)$ by
\[
\omega(s)\mapsto\mathrm{lim}_{0}\omega\mapsto\mathrm{lim}_{0}(\bar{\sigma}^{*}\omega)\mapsto(\bar{\sigma}^{*}\omega)(0)\mapsto=:\omega(0).
\]
Note that a section of $\cF^{1}(\cV^{\vee})_{e}$ lands in $F^{-p+1}H_{2p-r-1}(X_{0},\CC)$.
From Theorem \ref{limThm}(b) we have at once the
\begin{cor}
Given $\omega(s)\in\Gamma(\Delta,\cF^{1}(\cV^{\vee})_{e})$ and $\Xi\in \CH^{p}(\cX,r)$,
there exist lifts $\tilde{\nu}$ of $\nu_{\Xi^{*}}$ to $\cV_{e}$
that make $F_{\Xi,\omega}(s):=\langle\tilde{\nu}(s),\omega(s)\rangle$
holomorphic and single-valued on $\Delta$. For any such lift, we
have\begin{equation}\label{limTNF}\lim_{s\to 0} F_{\Xi,\omega}(s)\equiv \langle \mathrm{AJ}_{X_0}(\imath_{X_0}^*\Xi),\omega(0)\rangle
\end{equation}modulo periods of $\omega(0)$ over $H^{2p-r-1}(X_{0},\QQ(p))$.
\end{cor}
Of course, this limiting value may lie in $\CC$ modulo some horrible
subgroup with lots of generators. This corollary is used most successfully
when one has a splitting
\[
H^{2p-r-1}(X_{0})(p)\overset{\eta}{\twoheadrightarrow}\QQ(p)\;\;\;\;\;\text{[dually }\QQ(0)\overset{\eta^{\vee}}{\hookrightarrow}H_{2p-r-1}(X_{0})\text{]}
\]
of the MHS on the singular fiber, with $\omega(0)=\eta^{\vee}(1)$:
then \eqref{limTNF} becomes
\[
\lim_{s\to0}F(s)\equiv J(\eta)\left(\mathrm{AJ}_{X_{0}}(\imath_{X_{0}}^{*}\Xi)\right)\in J(\QQ(p))\cong\CC/\QQ(p).
\]
The tempered Laurent polynomials of \cite{DoranKerr} give one method
of constructing such splittings, for maximal unipotent degenerations
of Calabi-Yau varieties.%
\footnote{In the special case where $\cV$ is a VHS of CY type, and $\omega$
is a section of the top Hodge filtrand, $F_{\Xi,\omega}$ is called
a \emph{truncated normal function}.%
}
\begin{example}
Consider the Fermat quintic family defined by
\[
f(t,\underline{\x}):=t\sum_{i=0}^{4}\x_{i}^{5}-\prod_{i=0}^{4}\x_{i}=0
\]
in $\mathbb{P}^{4}$ ($t$ in a small disk about $0$). Let $\cX_{\Delta}$
be its semistable reduction. (See \cite{GGK} for an explicit description;
$X_{0}$ is a union of $4$ $\mathbb{P}^{3}$'s blown up along Fermat
quintic curves.) Then the standard residue $(3,0)$-form $\{\omega_{t}\}_{t\in\Delta}$
produces a splitting $\mathbb{Q}(0)\hookrightarrow H_{3}(X_{0})$
over $\{0\}$, essentially because $f(t,\underline{\x})/\prod_{i=0}^{4}\x_{i}$
is tempered \cite{DoranKerr}. In \cite{GGK}, this was used to study
limits of usual normal functions (paired with $\omega$) in $\CC/\QQ(2)$.
\end{example}
Of course, there are many cases where $H_{2p-r-1}(X_{0})$ (or at
least its image by $\mathfrak{r}^{*}$) \emph{is} $\QQ(0)$, and here
the Corollary applies automatically; for examples, see \cite{JW}
and \cite{dS}.

\section{\bf Application to a conjecture from topological string theory} \label{S6}

In this section we apply Theorem \ref{limThm}(b) (or \eqref{ssdThm})
to an an algebraic $K_{2}$-class on a 2-parameter family $\cX$ of
genus-2 curves. The fibers $X_{z_{1},z_{2}}$ of our family are obtained
by compactifying
\[
Y_{z_{1},z_{2}}:=\{\phi(\x,\y)=0\}\subset\left(\CC^{*}\right)^{2}
\]
in the toric Fano surface $\PP_{\Delta}$ associated to the Newton
polytope $\Delta=\Delta(\phi)$, where 
\[
\phi(\x,\y):=x_{0}+x_{1}\x+x_{2}\y+x_{3}\x^{-1}\y^{-1}+x_{4}\x^{-2}\y^{-2}
\]
and
\[
z_{1}=\frac{x_{1}x_{2}x_{3}}{x_{0}^{3}}\;,\;\;\;\;\; z_{2}=\frac{x_{0}x_{4}}{x_{3}^{2}}.
\]
For the total space $\cX$ (resp. $\cY$), we take the union of the
$X_{\uz}$ (resp. $Y_{\uz}$) for $\uz\in(\PP_{z_{1}}^{1}\backslash\{z_{1}=0\})\times(\PP_{z_{2}}^{1}\backslash\{z_{2}=0\})$;
note that the base contains the ``conifold point'' 
\[
\underline{z}^{(0)}:=(z_{1}^{(0)},z_{2}^{(0)}):=\left(-\frac{1}{25},\frac{1}{5}\right).
\]
 (This is actually an ordinary double-point of the conifold \emph{curve}.)
In effect, we will be applying the Theorem to a 1-parameter slice
through this point, which is a 1-parameter semistable degeneration.

We shall begin by describing two vanishing cycles $\alpha_{1},\alpha_{2}\in H_{1}(X_{\uz},\ZZ)$,
corresponding respectively to $z_{1}=0$ and $z_{2}=0$. Fix a small
$\epsilon>0$. For the cycle $\alpha_{1}$ vanishing at $z_{1}=0$,
we reason that $z_{1}\to0$ with $z_{2}$ constant corresponds to
$x_{1}\to0$ (or $x_{2}\to0$); let $\alpha_{1}$ be the cycle pinched
to the node at $x_{1}=0$. If we make the coordinate change $u=\x^{-1}\y$,
$v=\y^{-1}$, then
\[
\phi=x_{0}+\phi_{1}:=x_{0}+\left\{ x_{1}u^{-1}v^{-1}+x_{2}v^{-1}+x_{3}uv^2 + x_{4}u^{2}v^{4}\right\} 
\]
and (for very small $|x_{1}|,|x_{2}|$) the image of $\alpha_{1}$
under $Tube:\, H_{1}(X)\to H_{2}(\PP_{\Delta}\backslash X)$ (dual
to $2\pi\ay Res$) is given by $\tau_{1}=\{|u|=|v|=\epsilon\}$. Similarly,
$z_{2}\to0$ and $z_{1}$ constant corresponds to $x_{4}\to0$. Taking
$\alpha_{2}$ to be the cycle pinched to the node there, the coordinate
change $\tilde{u}=\x^{3}\y^{2}$, $\tilde{v}=\x^{-2}\y^{-1}$ makes
\[
\x\y\phi=\tilde{u}\tilde{v}\phi=x_{3}+\phi_{2}:=x_{3}+\left\{ x_{0}\tilde{u}\tilde{v}+x_{1}\tilde{v}^{-1}+x_{2}\tilde{u}^{3}\tilde{v}^{4}+x_{4}\tilde{u}^{-1}\tilde{v}^{-1}\right\} ;
\]
and (for very small $|x_{4}|,|x_{1}|$) $Tube(\alpha_{2})=\tau_{2}:=\{|\tilde{u}|=|\tilde{v}|=\epsilon\}$.
It should be emphasized that \emph{both} cycles vanish at $\uz=\underline{0}$,
but (as we describe below) \emph{neither} cycle vanishes at $\uz=\uz^{(0)}$.

By rescaling $\phi,\x,\y,\epsilon$, etc., we may both retain the
descriptions $Tube(\alpha_{i})=\tau_{i}$ and have $x_{1}=x_{2}=x_{4}=1$,
so that $\phi$ is tempered (and $z_{1}=x_{3}/x_{0}^{3}$, $z_{2}=x_{0}/x_{3}^{2}$).
This implies that the symbol $\{\x,\y\}\in \CH^{2}(\cY,2)$ lifts to
a class $\Xi\in \CH^{2}(\cX,2)$. The images of the $\Xi_{\underline{z}}:=\imath_{X_{\underline{z}}}^{*}\Xi$
under the Abel-Jacobi maps $$\mathrm{AJ}:\, \mathrm{CH}^2(X_{\underline{z}},2)\to H^1 (X_{\underline{z}},\mathbb{C}/\mathbb{Q}(2) )$$for 
\[
\underline{z}\in U:=\left\{ (z_{1},z_{2})\,|\,0<|z_{1}|<\tfrac{1}{25},\;0<|z_{2}|<\tfrac{1}{5}\right\} 
\]
may be computed as in \cite[\S 4.2]{DoranKerr} for elliptic curves,
suitably modified for genus 2 and two vanishing cycles. We now briefly
sketch the procedure, using the regulator current notation of \cite[$\S$1]{DoranKerr}.\footnote{In brief, we have $R\{f,g\}=\log(f)\tfrac{dg}{g} - 2\pi\ay \log(g)\delta_{T_f}$ and $R\{f,g,h\}=\log(f)\tfrac{dg}{g}\wedge\tfrac{dh}{h} + 2\pi\ay \log(g)\tfrac{dh}{h}\cdot \delta_{T_f} + (2\pi\ay)^2\log(h)\delta_{T_f \cap T_g}$, where $T_f = f^{-1}(\mathbb{R}_{<0})$ (oriented from $-\infty$ to $0$) and $\log(f)$ is the (discontinuous) branch with imaginary part in $[-\pi,\pi)$.}

Referring to the toric coordinate changes above, note the equality
of symbols $\{u,v\}=\{\x,\y\}=\{\tilde{u},\tilde{v}\}$ in $K_{2}(\mathbb{G}_{m}^{2})$,
hence in $\CH^{2}(Y_{\uz},2)$ (for $Y_{\uz}$ smooth). By temperedness,%
\footnote{Otherwise there would be a contribution from $Res_{v=0}R\{\phi,u,v\}$,
and not just the one shown (from $Res_{\phi=0}$); the detailed argument
is exactly as in \cite[\S 4.2]{DoranKerr}.%
} for sufficiently small nonzero $|z_{1}|,|z_{2}|$ we have \begin{align*}
\frac{1}{(2\pi\ay)^2} \int_{\tau_1} R\{ \phi,u,v\} &\underset{\QQ(1)}{\equiv}\frac{1}{2\pi\ay}\int_{\alpha_1} R\{u,v\}\underset{\QQ(1)}{\equiv}\frac{1}{2\pi\ay}\int_{\alpha_1} R\{\x,\y\}
\\
&=\frac{1}{2\pi\ay}\mathrm{AJ}(\Xi_{\uz})(\alpha_1)\; ;
\end{align*}and similarly\begin{equation*}
\frac{1}{(2\pi\ay)^2} \int_{\tau_2} R\{ \tilde{u}\tilde{v}\phi,\tilde{u},\tilde{v}\} \underset{\QQ(1)}{\equiv}\frac{1}{2\pi\ay}\int_{\alpha_2} R\{\tilde{u},\tilde{v}\} \underset{\QQ(1)}{\equiv}\frac{1}{2\pi\ay}\mathrm{AJ}(\Xi_{\uz})(\alpha_2)\; .
\end{equation*}Moreover, for small $\arg(z_{1})$ and $\arg(z_{2})$ we have $T_{\phi}\cap\tau_{1}=\emptyset$,
and so $R\{\phi,u,v\}=\log\phi\frac{du}{u}\wedge\frac{dv}{v}$. Noting
that $z_{1}^{2}z_{2}=x_{0}^{-5}$, this yields \begin{equation}\label{eq1A}\tfrac{1}{2\pi\ay} \mathrm{AJ}(\Xi_{\underline{z}})(\alpha_1)  \underset{\QQ(1)}{\equiv} \tfrac{1}{(2\pi\ay)^2} \int_{\tau_1}\log(x_0 + \phi_1)\frac{du}{u}\wedge\frac{dv}{v} 
\end{equation}\begin{equation}\label{eq1B}\underset{\QQ(1)}{\equiv}\log(x_0) - \sum_{n>0} (-1)^n \frac{x_0^{-n}}{n} \left[ \left( x_1 u^{-1}v^{-1} +x_2 v^{-1} + x_3 uv^2 +x_4 u^2 v^4 \right)^n \right]_{\underline{0}} 
\end{equation}\begin{equation}\label{eq1C}= -\frac{1}{5}\log(z_1^2 z_2)-\sum_{m,r\geq 0}{}'\frac{(5m+3r)!(-z_1)^r (-z_1^2 z_2)^m}{((2m+r)!)^2 m! r! (5m+3r)} \; ,
\end{equation}where $[\;\;]_{\underline{0}}$ takes the constant term of a Laurent
polynomial, and $\sum'$ means to omit $(m,r)=(0,0)$. For $\alpha_{2}$,
the analogous computation is \begin{equation}\label{eq2A}\tfrac{1}{2\pi\ay} \mathrm{AJ}(\Xi_{\underline{z}})(\alpha_2)  \underset{\QQ(1)}{\equiv} \tfrac{1}{(2\pi\ay)^2} \int_{\tau_2}\log(x_3 +\phi_2)\frac{d\tilde{u}}{\tilde{u}}\wedge\frac{d\tilde{v}}{\tilde{v}} 
\end{equation}\begin{equation}\label{eq2B}\underset{\QQ(1)}{\equiv}\log(x_3) - \sum_{n>0} (-1)^n \frac{x_3^{-n}}{n} \left[ \left( x_0 \tilde{u}\tilde{v}+ x_1\tilde{v}^{-1} +x_2 \tilde{u}^3\tilde{v}^4 + x_4 \tilde{u}^{-1}\tilde{v}^{-1}  \right)^n \right]_{\underline{0}} 
\end{equation}\begin{equation}\label{eq2C}= -\frac{1}{5}\log(z_1 z_2^3)-\sum_{m,r\geq 0} {}' \frac{(5m+2r)!(-z_1 z_2^3)^m (-z_2)^r}{(3m+r)! r! (m!)^2 (5m+2r)} \; .
\end{equation}The series in \eqref{eq1C} and \eqref{eq2C} converge absolutely
on $U$, hence compute $\frac{1}{2\pi\ay}\mathrm{AJ}(\Xi_{\uz})(\alpha_{i})$
($i=1,2$) there, and can be shown to converge to their limit at $\uz=\uz^{(0)}$.
Write $N_{i}=\log T_{i}$ for the monodromy logarithms about the 2
local components of the discriminant locus at $\uz^{(0)}$, and $N:=N_{1}+N_{2}$.
Then $\alpha_{1}$ and $\alpha_{2}$ generate $\text{coker}(N)\cong(\ker(N))^{\vee}$,
hence \eqref{eq1C} and \eqref{eq2C} are sufficient to capture the
limit of the normal function $\nu$ associated to $\Xi$ at $\uz^{(0)}$.

Turning to the right-hand side of \eqref{ssdThm}, we may write the
formula for the limiting curve $X_{\uz^{(0)}}$ as\begin{equation}\label{eq3}0 = \x + \y + \x^{-2}\y^{-2} -5 \x^{-1}\y^{-1} + 5 .
\end{equation}The two singularities of this curve are
\[
q_{1}=(-\varphi,-\varphi)\,,\;\;\;\;\; q_{2}=(-\tilde{\varphi},-\tilde{\varphi})\,,
\]
where $\varphi:=\tfrac{1}{2}(1+\sqrt{5})$ and $\tilde{\varphi}:=\tfrac{1}{2}(1-\sqrt{5})$.
The cycles $\gamma_{1},\gamma_{2}$ passing through these nodes \[\includegraphics[scale=0.55]{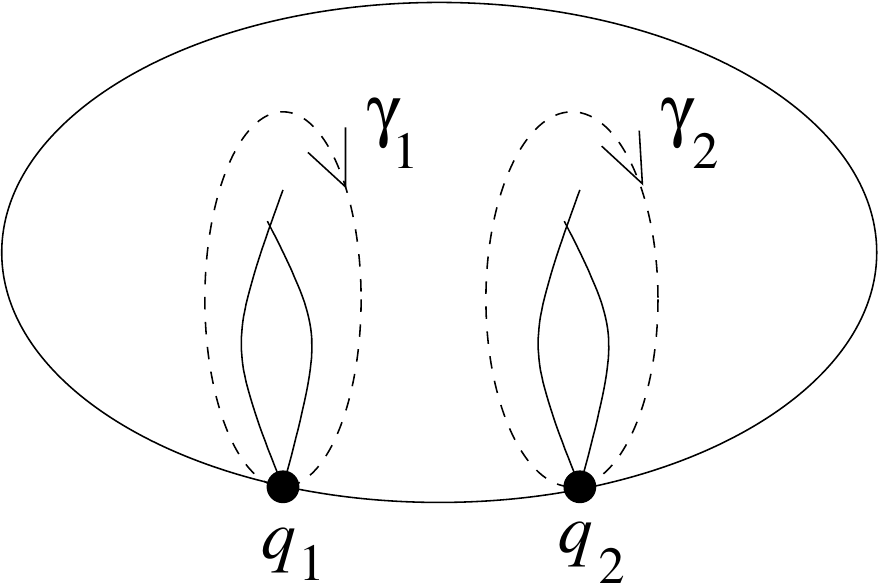}\]are
the images of $\alpha_{1}$ and $\alpha_{2}$ in $H_{1}(X_{\uz^{(0)}})$
under $\mathfrak{r}_{*}$. Consider the two uniformizations of $X_{\uz^{(0)}}$
by $\PP^{1}$:\begin{equation}\label{eq4}\x_1(t) = -\varphi\frac{\left( 1-\frac{\zeta^2}{t} \right)^3}{\left( 1-\frac{1}{\zeta^2 t}\right)^2 \left( 1-\frac{1}{t} \right) }\; ,\;\;\;\;\; \y_1 (t) = -\varphi\frac{\left( 1-\zeta^2 t \right)^3}{\left( 1-\frac{t}{\zeta^2}\right)^2 \left( 1-t \right) } \; ,
\end{equation}and\begin{equation}\label{eq5}\x_2(t) = -\tilde{\varphi}\frac{\left( 1-\frac{\zeta}{t} \right)^3}{\left( 1-\frac{1}{\zeta t}\right)^2 \left( 1-\frac{1}{t} \right) }\; ,\;\;\;\;\; \y_2 (t) = -\tilde{\varphi}\frac{\left( 1-\zeta t \right)^3}{\left( 1-\frac{t}{\zeta^2}\right)^2 \left( 1-t \right) } \; ,
\end{equation}where $\zeta:=e^{\frac{2\pi\ay}{5}}$. The first one $t\mapsto(\x_{1}(t),\y_{1}(t))$
maps $t=0,\infty$ to $q_{1}$; the second maps $0,\infty\mapsto q_{2}$:
so they send the path from ``$-\infty$ to $0$'' to $\gamma_{1}$
resp. $\gamma_{2}$. This allows us to ``plug in'' to the formula
from \cite[\S 6.2]{DoranKerr}, which assigns a divisor $\mathcal{N}$
on $\PP^{1}\backslash\{0,\infty\}$ to each uniformization. In the
present case,
\[
\mathcal{N}_{2}=-6[\zeta]+9[\zeta^{2}]+4[\zeta^{3}]+4[\zeta^{4}]
\]
and
\[
\mathcal{N}_{1}=-6[\zeta^{2}]+9[\zeta^{4}]+4[\zeta]+4[\zeta^{3}].
\]
Working modulo the scissors congruence relations
\[
[\xi]+[\tfrac{1}{\xi}]=0,\;[\xi]+[\bar{\xi}]=0,\;[\xi]+[1-\xi]=0,\;\text{and}
\]
\[
[x]+[y]+[\tfrac{1-x}{1-xy}]+[1-xy]+[\tfrac{1-y}{1-xy}]=0,
\]
we have\begin{equation}\label{eq6}\left\{ \begin{array}{c}\mathcal{N}_{1}\equiv-10[\zeta^{2}]+5[\zeta^{4}]\equiv10[-\zeta\tilde{\varphi}]\equiv10[\zeta\varphi]\\\mathcal{N}_{2}\equiv-10[\zeta]+5[\zeta^{2}]\equiv10[-\zeta^{3}\varphi]\equiv10[e^{\frac{\pi\ay}{5}}\varphi] \; . \end{array}\right.
\end{equation}But according to {[}loc.cit.{]} we then have (using \eqref{eq6})\begin{equation}\label{eq7}\text{Re}\left(\tfrac{1}{2\pi\ay} \mathrm{AJ}(\Xi_{\uz^{(0)}})(\gamma_1) \right) = \tfrac{1}{2\pi} D_2 (\mathcal{N}_1) = \tfrac{5}{\pi} D_2 (\zeta\varphi)\; ,
\end{equation}\begin{equation}\label{eq8}\text{Re}\left(\tfrac{1}{2\pi\ay} \mathrm{AJ}(\Xi_{\uz^{(0)}})(\gamma_2) \right) = \tfrac{1}{2\pi} D_2 (\mathcal{N}_2) = \tfrac{5}{\pi} D_2 (e^{\frac{\pi\ay}{5}}\varphi)\; .
\end{equation}Here $\Xi_{\uz^{(0)}}$ denotes the pullback motivic cohomology class
on $X_{\uz^{(0)}}$, and $D_2(z)=\mathrm{Im}(\mathrm{Li}_2(z))+\arg(1-z)\log|z|$ is the Bloch-Wigner function.

By \eqref{ssdThm}, we have that \eqref{eq7} {[}resp. \eqref{eq8}{]}
is equal to the real part of the $\uz\to\uz^{(0)}$ limit of \eqref{eq1C}
{[}resp. \eqref{eq2C}{]}, which yields precisely the relations \begin{equation}\label{eq9}\tfrac{5}{\pi}D_2(e^{\frac{2\pi\ay}{5}}\varphi) = \log(5) - \sum_{m,r\geq 0}{}'\frac{(-1)^m (5m+3r)!}{((2m+r)!)^2 m! (5m+3r) 5^{5m+2r}}
\end{equation}and\begin{equation}\label{eq10}\tfrac{5}{\pi}D_2(e^{\frac{\pi\ay}{5}}\varphi) = \log(5) - \sum_{m,r\geq 0}{}'\frac{(-1)^r (5m+2r)!}{(3m+r)! r! (m!)^2 (5m+2r)5^{5m+r}}
\end{equation} conjectured by Codesido, Grassi and Marino \cite[(4.106)]{Ma} as
a test (for the mirror $\CC^{3}/\ZZ^{5}$ geometry) of the correspondence
between spectral theory and enumerative geometry proposed in \cite{GHM}.

\end{document}